\documentclass[10pt]{amsart}

\usepackage[margin=1in]{geometry}
\usepackage{amsmath}
\usepackage{float}
\usepackage{booktabs}
\usepackage[dvipsnames]{xcolor} 

\usepackage{fix-cm}

\usepackage{wasysym}

\usepackage{blkarray}

\usepackage{amscd,amsmath,amssymb,amsfonts,amsthm}
\usepackage{enumerate,mathrsfs,stmaryrd,latexsym, comment, mathtools, mathdots} 

\usepackage[mathscr]{euscript}

\usepackage{caption}
\usepackage{subcaption}
\usepackage{hyperref}

\usepackage[all]{xy}

\usepackage{tikz-cd}
\usepackage{tikz}
\usetikzlibrary{snakes, 3d, matrix,decorations.pathreplacing,calc,arrows,decorations.pathmorphing, patterns}
\usetikzlibrary {positioning}
\usetikzlibrary{calc,backgrounds}

\pgfdeclarelayer{background}
\pgfdeclarelayer{foreground}
\pgfsetlayers{background,main,foreground}

\usepackage{tikz-3dplot}

\numberwithin{equation}{section}
\allowdisplaybreaks[1]

\makeatletter
\newcommand{\leqnomode}{\tagsleft@true\let\veqno\@@leqno}
\newcommand{\reqnomode}{\tagsleft@false\let\veqno\@@eqno}
\makeatother

\DeclareMathOperator{\Gr}{Gr}

\DeclareMathOperator{\Hom}{Hom}

\DeclareMathOperator{\Sp}{Sp}
\DeclareMathOperator{\Spin}{Spin}

\DeclareMathOperator{\Aut}{Aut}

\DeclareMathOperator{\Ric}{Ric}
\DeclareMathOperator{\SGr}{SGr}

\newtheorem{theorem}{Theorem}[section]
\newtheorem{lemma}[theorem]{Lemma}
\newtheorem{proposition}[theorem]{Proposition}
\newtheorem{corollary}[theorem]{Corollary}

\newtheorem{Question}[theorem]{Question}

\theoremstyle{definition}

\newtheorem{definition}[theorem]{Definition}
\newtheorem{remark}[theorem]{Remark}

\begin{document}
	
\title[Greatest Ricci lower bounds of projective horospherical manifolds]{Greatest Ricci lower bounds of projective horospherical manifolds of Picard number one}

\author{DongSeon Hwang}
\address{Center for Complex Geometry, Institute for Basic Science (IBS), Daejeon 34126, Republic of Korea}
\email{dshwang@ibs.re.kr}

\author{Shin-young Kim}
\address{Center for Geometry and Physics, Institute for Basic Science (IBS), Pohang 37673, Republic of Korea}
\email{shinyoungkim@ibs.re.kr}

\author{Kyeong-Dong Park}
\address{Department of Mathematics and Research Institute of Natural Science, Gyeongsang National University, Jinju 52828, Republic of Korea}
\email{kdpark@gnu.ac.kr}



\begin{abstract}
A horospherical variety is a normal $G$-variety such that a connected reductive algebraic group $G$ acts with an open orbit isomorphic to a torus bundle over a rational homogeneous manifold. 
The projective horospherical manifolds of Picard number one are classified by Pasquier, and it turned out that the automorphism groups of all nonhomogeneous ones are non-reductive, which implies that they admit no K\"{a}hler--Einstein metrics. 
As a numerical  measure of the extent to which a Fano manifold is close to be K\"{a}hler--Einstein,  
we compute the greatest Ricci lower bounds of projective horospherical manifolds of Picard number one using the barycenter of each moment polytope with respect to the Duistermaat--Heckman measure based on a recent work of Delcroix and Hultgren. 
In particular, the greatest Ricci lower bound of the odd symplectic Grassmannian $\text{SGr}(n,2n+1)$ can be arbitrarily close to zero as $n$ grows.
\end{abstract}

\maketitle
\setcounter{tocdepth}{1}
\date{\today}


\section{Introduction}

One of the most important problems in K\"ahler geometry is to find a K\"ahler--Einstein metric. 
Whereas Calabi--Yau manifolds and K\"{a}hler manifolds of general type always admit K\"{a}hler--Einstein metrics by the works of Aubin~\cite{aubin78} and Yau~\cite{yau78}, Fano manifolds do not necessarily admit K\"{a}hler--Einstein metrics in general. 
By the Yau--Tian--Donaldson conjecture, now completely solved for anti-canonically polarized Fano varieties (\cite{cds1, cds2, cds3},  \cite{tian15} and \cite{LXZ}), a Fano variety admits a K\"{a}hler--Einstein metric if and only if it is K-polystable. 
 
The \emph{greatest Ricci lower bound} $R(X)$ of a Fano manifold $X$ is defined as
$$ R(X):=\sup \{  0 \leq t \leq 1 \colon \text{there exists a K\"{a}hler form } \omega \in c_1(X) \text{ with } \Ric(\omega) \geq t \, \omega \}.$$
This invariant was first studied by Tian~\cite{tian92}, and was explicitely defined by Rubinstein~\cite{Rubinstein08, Rubinstein09}, where it was called Tian's $\beta$-invariant. 
It was further studied by Sz\'{e}kelyhidi~\cite{Sze11}; Song and Wang~\cite{SW16}. 
Note that if $X$ admits a K\"{a}hler--Einstein metric then the greatest Ricci lower bound $R(X)$ is equal to one. 
Thus $R(X)$ can be regarded as a measure of the extent to which a Fano manifold is close to be K\"{a}hler--Einstein. 
More precisely, $R(X)$ is shown to be the same as the maximum existence time $t \in [0, 1]$ of Aubin and Yau's continuity path: 
for a K\"{a}hler form $\omega \in c_1(X)$, we want to find a K\"{a}hler form $\omega_t$ satisfying the equation 
$$\Ric(\omega_t) = t \, \omega_t + (1 - t) \omega$$
depending on a parameter $t \in [0, 1]$, which for $t=1$ gives the Einstein equation. 
 
The greatest Ricci lower bounds have been considered in certain subclasses of Fano manifolds. 
For any toric Fano manifold $X$, Li~\cite{Li11} found an explicit formula for the greatest Ricci lower bound $R(X)$ purely in terms of the moment polytope associated to $X$. 
This result is extended to Fano manifolds with torus actions of complexity one by Cable~\cite{Cable19} and homogeneous toric bundles by Yao~\cite{yao17}. 
When $X$ is a smooth Fano equivariant compactifications of complex Lie groups, Delcroix~\cite{Del17} obtained a formula for the greatest Ricci lower bound, where the barycenter of the moment polytope with respect to the Lebesgue measure is replaced by the barycenter with respect to the Duistermaat--Heckman measure. 
Furthermore, the recent result of Delcroix and Hultgren \cite{DH21} extend the formula to the case of \emph{horosymmetric} manifolds introduced in \cite{Del20horosymmetric}, which is a class of spherical varieties including horospherical manifolds and smooth symmetric varieties. 
Since any $\mathbb Q$-Fano spherical variety admits a special test configuration with horospherical central fiber by \cite[Corollary 3.31]{Del20}, for the Donaldson--Futaki invariants of $\mathbb Q$-Fano spherical varieties, it is enough to calculate them in the case of horospherical varieties.

The purpose of this paper is to compute the greatest Ricci lower bounds of projective horospherical manifolds, i.e. smooth projective horospherical varieties, of Picard number one. 
Recall that the value is always one for rational homogeneous manifolds because they are K\"{a}hler--Einstein (see \cite[Section 5]{mat72}). 
The horospherical manifolds are the simplest examples of spherical varieities including both toric manifolds and rational homogeneous manifolds.  All nonhomogeneous projective horospherical manifolds of Picard number one are completely classified by Pasquier.

\begin{theorem}[{\cite[Theorem 0.1]{Pasquier}}] 
\label{horosphercal} 
Let $X$ be a projective horospherical manifold of Picard number one. 
If $X$ is nonhomogeneous, it is of rank one and its automorphism group is a connected non-reductive linear algebraic group. 
Moreover, $X$ is uniquely determined by its two closed $G$-orbits $Y$ and $Z$, isomorphic to rational homogeneous manifolds $G/P^{\alpha_i}$ and $G/P^{\alpha_j}$, respectively, where $(G, \alpha_i, \alpha_j)$ is one of the following: 
\begin{enumerate}
\item[\rm (1)] $X^1(n) := (B_n, \alpha_{n-1}, \alpha_n)$ with $n \geq 3$;
\item[\rm (2)] $X^2 := (B_3, \alpha_1, \alpha_3)$;
\item[\rm (3)] $X^3(n,k) := (C_n, \alpha_{k}, \alpha_{k-1})$
 with $n \geq k \geq 2$;
\item[\rm (4)] $X^4 := (F_4, \alpha_2, \alpha_3)$;
\item[\rm (5)] $X^5 := (G_2, \alpha_2, \alpha_1)$.
\end{enumerate}
\end{theorem}

These are two-orbit varieties, with one open orbit and one closed orbit, under the action of their automorphism groups $\Aut(X)$ which are non-reductive, 
and their blow-ups along the closed orbit are again two-orbit varieties with respect to $\Aut(X)$.  
See Subsection~\ref{Projective horospherical manifolds of Picard number one} for further  interesting geometry of these nonhomogeneous horospherical manifolds. In particular, it is known that $X^3(n,k) = \SGr(k,2n+1)$, the odd symplectic Grassmannians. 

Contrary to smooth projective \emph{symmetric} varieties of Picard number one (see \cite{LPY21}), 
all nonhomogeneous projective horospherical manifolds of Picard number one admit no K\"{a}hler--Einstein metrics by a theorem of  Matsushima in \cite{matsushima} since their automorphism groups are not reductive. 
On the other hand, by \cite[Theorem~4.1]{Del20YTD} and  \cite[Theorem~1.2]{Kanemitsu_manuscripta} 
any smooth Fano two-orbit varieties of Picard number one whose blow-ups along the closed orbit are again two-orbit varieties admit K\"{a}hler--Einstein metrics except for horospherical varieties. 

\begin{theorem}
\label{Main theorem} 
For each nonhomogeneous projective horospherical manifold $X$ of Picard number one, its greatest Ricci lower bound $R(X)$ is as in Table \ref{table1}, in which $n \geq 3$ for $X^1(n)$ and $n \geq k \geq 2$ for $X^3(n,k)$.
 \end{theorem}

\vskip -1em

\begin{table}[h!]
\begin{center}
\begin{tabular}{|c |c |c|}
	\hline 
	$X$ & $\dim X$  & $R(X)$ \\
		\hline \hline
	$X^1(n)$ 	&	$\displaystyle \frac{n(n+3)}{2}$   	&	$\displaystyle \frac{n \displaystyle \int_{-n}^{2} (2-t) (n+t)^{n-1} (t+2n+2)^{\frac{n(n-1)}{2}} \, dt}{\displaystyle \int_{-n}^{2} (2-t) (n+t)^{n} (t+2n+2)^{\frac{n(n-1)}{2}} \, dt}$ 
\\   
	$X^2$	&	$9$  	&	$\displaystyle \frac{20}{21} \approx 0.952$  
\\ && \\
	$X^3(n,n)$	&	$\displaystyle \frac{n(n+3)}{2}$  	&	$
	\displaystyle \frac{2 \times (2n+1)!}{(n+2)(2^n \times n !)^2}$  
\\ && \\
  $X^3(n,k)$	&	$\displaystyle \frac{k(4n-3k+3)}{2}$ 	&	$\displaystyle \frac{(2n-2k+2) \displaystyle \int_{-k}^{2n-2k+2} (k+t)^{k-1} (2n-2k+2-t)^{2n-2k+1} (4n-3k+4-t)^{k-1} \, dt}{\displaystyle \int_{-k}^{2n-2k+2} (k+t)^{k-1} (2n-2k+2-t)^{2n-2k+2} (4n-3k+4-t)^{k-1} \, dt}$
\\
	$X^4$	&	$23$ 	&	$\displaystyle \frac{178992099}{243545402} \approx 0.734$ 
\\  && \\
	$X^5$	&	$7$ 	&	$\displaystyle \frac{56}{67} \approx 0.8358$ 
\\  && \\
		\hline 
\end{tabular}
\caption{Dimensions and the greatest Ricci lower bounds.}
\label{table1}
\end{center}
\end{table}

In the case of $X^1(n)$ and $X^3(n,k)$, it is possible to compute the exact number $R(X)$ once the integers $n$ and $k$ are given. 
For instance, see the following Table \ref{table2} for the approximate values of $R(X^1(n))$, $R(X^3(n, 2))$, $R(X^3(n, 3))$ and $R(X^3(n, 4))$ for some $n$.

\begin{table}[h]
\begin{center}
\begin{tabular}{c c c c c c c c c c c}
		\toprule
		$n$ & $3$   & $4$ & $5$ & $6$ & $7$ & $10$ & $20$ & $30$ & $50$ & $70$
\\
		\midrule 
	$R(X^1(n))$	&	$0.8955$	&	$0.8755$  &	$0.8686$ 	&	$0.8685$ 	&	$0.8715$	& 	$0.8863$  &	$0.9251$ 	&	$0.9451$ 	&	$0.9644$ 	&	$0.9737$
\\
		\midrule 
	$R(X^3(n, 2))$	&	$0.972$	&	$0.984$  &	$0.99$ 	&	$0.993$ 	&	$0.995$	& 	$0.9975$  &	$0.9994$ 	&	$0.9997$ 	&	$0.9999$ 	&	$0.99995$
\\
		\midrule 
	$R(X^3(n, 3))$	&	$0.875$	&	$0.9375$  &	$0.9625$ 	&	$0.975$ 	&	$0.982$	& 	$0.9917$  &	$0.9980$ 	&	$0.9991$ 	&	$0.9997$ 	&	$0.99984$
\\
		\midrule 
	$R(X^3(n, 4))$	&	$-$	&	$0.820$  &	$0.902$ 	&	$0.938$ 	&	$0.958$	& 	$0.9813$  &	$0.9958$ 	&	$0.9982$ 	&	$0.9994$ 	&	$0.99968$
\\
		\bottomrule
\end{tabular}
\caption{Approximate values of $R(X^1(n))$, $R(X^3(n, 2))$, $R(X^3(n, 3))$ and $R(X^3(n, 4))$.}
\label{table2}
\end{center}
\end{table}

We do not know how to further simplify the integrals even though we have found a nice expression in the case of $X^3(n,n)$. 
However, this complicated expression is still useful to find the following asymptotic behavior of the greatest Ricci lower bounds.

\begin{corollary} 
\label{limits of greatest Ricci lower bounds}
\begin{enumerate}
\item[\rm (1)] 
$R(X^1(n))$ converges to $1$ as $n$ increases, that is, $\displaystyle \lim_{n \to \infty} R(X^1(n)) = 1$. 
\item[\rm (2)] 
For a fixed integer $k \geq 2$, $R(X^3(n, k))$ converges to $1$ as $n$ increases, that is, $\displaystyle \lim_{n \to \infty} R(X^3(n, k)) = 1$. 
\item[\rm (3)] 
$R(X^3(n, n))$ converges to zero as $n$ increases, that is, $\displaystyle \lim_{n \to \infty} R(X^3(n, n)) = 0$. 
\end{enumerate}
\end{corollary}

Motivated by Corollary~\ref{limits of greatest Ricci lower bounds} (3), one may ask the following question suggested by a referee. 
See Corollary~\ref{lowerbound} for further discussions. 

\begin{Question}
\label{realization}
Let $t$ be a real number with $0 < t < 1$. Does there exist a sequence of Fano horospherical manifolds of Picard number one whose greatest Ricci lower bounds converge to $t$? 
\end{Question}

The greatest Ricci lower bound $R(X)$ of a Fano manifold $X$ is closely related with Tian's $\alpha$-invariant and the $\delta$-invariant defined by Fujita and Odaka \cite{FO18}, which plays an important role of verifying the existence of K\"{a}hler--{E}instein metrics for explicitly given Fano varieties. 
In particular, if $X$ does not admit a K\"{a}hler--{E}instein metric, then we have $R(X) = \delta(X, -K_X)$ by \cite[Theorem~5.7]{CRZ19} and \cite[Corollary~7.6]{BBJ21} (see also \cite[Section~3.3]{Golota20}), and $R(X) \geq \alpha(X) \cdot \displaystyle \frac{\dim X + 1}{\dim X}$ by \cite[Theorem~2.1]{tian87}.  
It is interesting to note that the greatest Ricci lower bound, the alpha-invariant and the delta-invariant can be arbitrarily close to $0$ for the odd symplectic Grassmannians  $X^3(n,n) = \SGr(n, 2n+1)$ by Corollary~\ref{limits of greatest Ricci lower bounds}~(3).  

\begin{corollary}[=Corollary \ref{limit of R(X3n)}] 
For every sufficiently small number $\epsilon > 0$, there exists a natural number $n$ such that $R(\SGr(n, 2n+1)) < \epsilon$, $\alpha(\SGr(n, 2n+1)) < \epsilon$ and $\delta(\SGr(n, 2n+1)) < \epsilon$.
\end{corollary} 
 
The paper is organized as follows.
In Section~2, we recall the theory of spherical varieties by focusing on horospherical varieties, and we state an explicit expression of the greatest Ricci lower bounds for horospherical manifolds in terms of moment polytopes obtained by Delcroix and Hultgren. 
In Section~3, we obtain the greatest Ricci lower bound of a projective horospherical manifold of Picard number one by computing the barycenter of its moment polytope with respect to the Duistermaat--Heckman measure.  


\section{Spherical varieties and horospherical varieties} 
Throughout the paper, we denote by $G$ a connected reductive algebraic group over $\mathbb{C}$ and by $X$ an irreducible normal variety over $\mathbb{C}$.

\subsection{Spherical varieties and colors} 
We review general notions and results about spherical varieties for our use. 
See \cite{Knop91}, \cite{Timashev11} and \cite{Gandini18} for the reference.

\begin{definition}
\label{spherical variety}
A normal variety $X$ equipped with an action of $G$ is called \emph{spherical} if a Borel subgroup $B$ of $G$ acts on $X$ with an open dense orbit.
\end{definition}

Let $X$ be a spherical variety, $B$ be a Borel subgroup of $G$, and $T$ be a maximal torus of $B$. 
Let $G/H$ be an open dense $G$-orbit of $X$.
For a character $\chi \in \mathfrak X(B)$, 
let $$\mathbb C(G/H)^{(B)}_{\chi} = \{ f \in \mathbb C(G/H) : b \cdot f = \chi(b) f \text{ for all } b \in B \}$$ be the set of $B$-semi-invariant functions in $\mathbb C(G/H)$ associated to $\chi$, where $\mathbb C(G/H)$ denotes the function field of $G/H$. 
Recall that  $\mathfrak X(B) = \mathfrak X(T)$ and  $\mathbb C(G/H) = \mathbb C(X)$ in our case. 

The \emph{spherical weight lattice} $\mathcal M$ of $G/H$ is defined as a subgroup of the character group $\mathfrak X(T)$ such that each element $\chi \in \mathcal M$ has the non-zero set of $B$-semi-invariant functions, that is, 
$$\mathcal M = \{ \chi \in \mathfrak X(T) : \mathbb C(G/H)^{(B)}_{\chi} \neq 0 \}.$$
Note that every function $f_{\chi}$ in $\mathbb C(G/H)^{(B)}$ is determined by its weight $\chi$ up to constant. 
This is because any $B$-invariant rational function on $X$ is constant, that is, $\mathbb C(G/H)^{B} = \mathbb C$. 
The spherical weight lattice $\mathcal M$ is a free abelian group of finite rank. 
We define the \emph{rank} of $G/H$ as the rank of the lattice $\mathcal M$. 
Let $\mathcal N = \Hom(\mathcal M, \mathbb Z)$ denote its dual lattice together with the natural pairing $\langle \, \cdot \, , \, \cdot \, \rangle \colon \mathcal N \times \mathcal M \to \mathbb Z$. 

It is known that every open $B$-orbit of a spherical variety is an affine variety and so its complement has pure codimension one consisting of finite $B$-stable prime divisors. 

\begin{definition}
\label{color}
For a spherical homogeneous space $G/H$, 
$B$-stable prime divisors in $G/H$ are called \emph{colors} of $G/H$. 
We denote by $\mathfrak D = \{ D_1, \cdots, D_k \}$ the set of colors of $G/H$. 
\end{definition}

For a spherical $G$-variety $X$, the colors of $X$ means $B$-stable but not $G$-stable divisors that contain a $G$-orbit in $X$. 
We denote by $\mathfrak D(X)$ the set of colors of $X$. 
We note that a color of a spherical variety $X$ corresponds to a color of the open $G$-orbit $G/H$ of $X$. 

As a $B$-semi-invariant function $f_{\chi}$ in $\mathbb C(G/H)^{(B)}_{\chi}$ is unique up to constant, 
we define the \emph{color map} $\rho \colon \mathfrak D \to \mathcal N$ by $\langle \rho(D), \chi \rangle = \nu_D(f_{\chi})$ for $\chi \in \mathcal M$, where $\nu_D$ is the discrete valuation associated to a divisor $D$, that is, $\nu_D(f)$ is the vanishing order of $f$ along $D$. 
Unfortunately, the color map is not injective in general.

Every discrete $\mathbb Q$-valued valuation $\nu$ of the function field $\mathbb C(G/H)$ induces a homomorphism $\hat{\rho}(\nu) \colon \mathcal M \to \mathbb Q$ defined by $\langle \hat{\rho}(\nu), \chi \rangle = \nu(f_{\chi})$. 
Hence, we get a map $$\hat{\rho} \colon \{ \text{discrete $\mathbb Q$-valued valuations on $G/H$} \} \to \mathcal N \otimes \mathbb Q.$$ 
Luna and Vust \cite{LV83} showed that the restriction of $\hat{\rho}$ to the set of all $G$-invariant discrete valuations on $G/H$ is injective. 
Let $\mathcal V$ be the set of $G$-invariant discrete $\mathbb Q$-valued valuations on $G/H$. 
Since the map $\hat{\rho}|_{\mathcal V}$ is injective, we may consider $\mathcal V$ as a subset of $\mathcal N \otimes \mathbb Q$. 
It is known that $\mathcal V$ is a full-dimensional (co)simplicial cone of $\mathcal N \otimes \mathbb Q$, which is called the \emph{valuation cone} of $G/H$. 
For a $G$-stable divisor $E$ in $X$, we simplify the notation $\hat{\rho}(\nu_E)$ as $\hat{\rho}(E)$ which is in $\mathcal V$. 

Since $B$ stabilizes the open $B$-orbit, the stabilizer $P$ of the open $B$-orbit under the $G$-action is a parabolic subgroup of $G$. 
Then, we have the algebraic Levi decomposition $P=L \ltimes P^u$, where $P^u$ is the unipotent radical of $P$ and $L$ is the Levi factor of $P$. 
Let $\Phi^+$ be the set of positive root of $G$, $\Phi_{L}$ be the set of roots of $L$ and $\Phi_{P^u}$ be the set of roots of the unipotent radical $P^u$.  
Let $\rho_G$ be the sum of all fundamental weights of the root system for $G$, $\rho_L$ be the sum of all fundamental weights of the root system for the Levi factor $L$ and let $\rho_P$ be the half of the sum of roots of $P$. Then we have 
$$2 \rho_P=\displaystyle \sum_{\alpha \in \Phi_{P^u}} \alpha = \sum_{\alpha \in \Phi^+ \backslash \Phi_{L}} \alpha= 2 \rho_G - 2 \rho_L.$$

\subsection{Horospherical varieties} 
\label{horospherical varieties}

A normal $G$-variety $X$ is \emph{horospherical}
if $G$ acts with an open orbit $G/H$ isomorphic to a torus bundle over a rational homogeneous manifold $G/P$, or equivalently,
if the isotropy subgroup $H$ of a general point contains the unipotent radical $U$ of a Borel subgroup $B$. 
In this case, $G/H$ above is called a \emph{horospherical homogeneous space} and $X$ is also called an \emph{embedding of a horospherical homogeneous space $G/H$}.   
The dimension of the torus fiber is called the \emph{rank} of a horospherical variety. 
A smooth horospherical variety is called a \emph{horospherical manifold}.
Note that every horospherical variety is spherical. 
Indeed, as the full flag variety $G/B$ contains an open $U$-orbit by the Bruhat decomposition, if $U \subset H$ then $G/B$ has an open $H$-orbit; hence $G/H$ contains an open $B$-orbit.  

The spherical data of a projective horospherical variety can be described as follows. 

\begin{proposition}[{\cite[Section~28]{Timashev11}}]
\label{spherical data of horospherical space} 
For a projective horospherical variety $X$ with an open orbit $G/H$, we denote by $\pi \colon X \to G/P$ the torus fibration map, 
where a parabolic subgroup $P = N_G(H)$ is the normalizer of $H$ in $G$. 
Then, 
\begin{itemize}
	\item the spherical weight lattice $\mathcal M$ is the lattice $\mathfrak X(T / T \cap H)$, which implies that the rank of $X$ is equal to the dimension of the torus fiber of $\pi$, 
	\item the valuation cone $\mathcal V$ is $\mathcal N \otimes \mathbb Q$,  
	\item $\mathfrak D(X) = \{ {\pi^{-1}(D_{\alpha}}) : D_{\alpha} \text{ is a Schubert divisor associated to a simple root $\alpha$ of } G/P \}$, and 
	\item the image $\rho({\pi^{-1}(D_{\alpha})})$ under the color map $\rho$ is the restriction $\alpha^{\vee}|_{\mathcal M}$ of the simple coroot associated to a Schubert divisor $D_{\alpha}$ of $G/P$. 
\end{itemize}
\end{proposition}

Moreover, the colors of $X$ are the closures $\overline{B w_0 s_{\alpha} P/H}$ in $X$ of $B$-stable irreducible divisors of $G/H$, where $w_0$ is the longest element in the Weyl group of $(G, T)$ and $s_{\alpha}$ is the simple reflection associated to a simple root $\alpha$ not in $P$. 

The anticanonical divisor $- K_X$ of a spherical variety $X$ is explicitly described in \cite[Theorem 4.2]{Brion97} and \cite[Section 3.6]{Luna97}.
We use the following special case  for projective horospherical varieties. 
 
\begin{proposition}[{\cite[Proposition 3.1]{Pasquier08}}] 
\label{expression of anticanonical divisor}
Let $X$ be a projective horospherical variety with an open orbit $G/H$. Denote by $-K_X$ the  anticanonical divisor. 
Then we have the following linear equivalence 
$$-K_X \sim   \sum_{i=1}^k \langle \alpha^{\vee}_i, 2 \rho_P \rangle D_i + \sum_{j=1}^{\ell} E_j$$ 
for colors $D_i \in \mathfrak D$ and $G$-stable divisors $E_j$ in $X$, where $\alpha^{\vee}_i$ is the coroot of simple root $\alpha_i$ of $G$. 
Here, the coefficient $\langle \alpha^{\vee}_i, 2 \rho_P \rangle$ of each color is positive.
\end{proposition}

Let $L$ be a $G$-linearized ample line bundle on a spherical $G$-variety $X$. 
By the multiplicity-free property of spherical varieties, the algebraic moment polytope $\Delta(X, L)$ encodes the structure of representation of $G$ in the spaces of multi-sections of tensor powers of $L$.

\begin{definition}
\label{definition of moment polytope}
The \emph{algebraic moment polytope} $\Delta(X, L)$ of $L$ with respect to $B$ is defined as 
the closure of $\bigcup_{k \in \mathbb N} \Delta_k / k$ in $\mathfrak X(T) \otimes \mathbb R$, 
where $\Delta_k$ is a finite set consisting of (dominant) weights $\lambda$ such that
\begin{equation*}
H^0(X, L^{\otimes k}) = \bigoplus_{\lambda \in \Delta_k} V_G(\lambda).
\end{equation*} 
Here, $V_G(\lambda)$ means the irreducible representation of $G$ with highest weight $\lambda$. 
\end{definition}

The algebraic moment polytope $\Delta(X, L)$ for a polarized (spherical) $G$-variety ($X$, $L$) was introduced by Brion in \cite{Brion87} as a purely algebraic version of the Kirwan polytope. 
This is indeed the convex hull of finitely many points in $\mathfrak X(T) \otimes \mathbb R$ (see \cite{Brion87}). 

In \cite[Chapter 3]{Pasquier08}, Pasquier studied anticanonically polarized Gorenstein Fano horospherical varieties $(X, K_X^{-1})$ and the corresponding algebraic moment polytopes $\Delta(X, K_X^{-1})$, so called $G/H$-reflexive polytopes. 

\begin{proposition}[{\cite[Proposition 3.4]{Pasquier08}}]
Let $G/H$ be a horospherical homogeneous space. 
There is a bijection between the set of isomorphism classes of Gorenstein Fano embeddings of $G/H$ and the set of unimodular equivalence classes of  $G/H$-reflexive polytopes in $\mathcal N \otimes \mathbb R$ up to natural isomorphisms.  
\end{proposition}
 
\begin{proposition}[{\cite{Brion89, Brion97, Pasquier08}}]
\label{moment polytope of spherical variety}
Let $X$ be a Gorenstein Fano horospherical variety. 
Then, assuming the notation in Proposition~\ref{expression of anticanonical divisor},   
the moment polytope $\Delta(X, K_X^{-1})$ is $G/H$-reflexive and 
$$\Delta(X, K_X^{-1}) = 2\rho_P + Q_X^*,$$
where the polytope $Q_X$ is the convex hull of the set 
\begin{equation*}
\left\{ \frac{\rho(D_i)}{\langle \alpha^{\vee}_i, 2 \rho_P \rangle} : i= 1, \cdots, k \right\} \cup \{ \hat{\rho}(E_j) : j= 1, \cdots , \ell \}
\end{equation*}  
in $\mathcal N \otimes \mathbb R$ and 
its dual polytope $Q_X^*$ is defined as 
$\{ m \in \mathcal M \otimes \mathbb R : \langle n, m \rangle \geq -1 \text{ for every } n \in Q_X \}$. 
\end{proposition}

With the description of the moment polytope, we can compute the greatest Ricci lower bound of a Fano horospherical manifold.  

\begin{proposition}[{\cite[Corollary 1.3]{DH21}}] 
\label{Greatest Ricci lower bounds}
The greatest Ricci lower bound $R(X)$ of a Fano horospherical manifold $X$ is equal to 
$$\sup \left \{ t \in (0, 1) : 2 \rho_{P} + \frac{t}{1-t} (2 \rho_{P} - \textbf{bar}_{DH}(\Delta)) \in \textup{Relint}(\Delta) \right \},$$ 
where $\textup{Relint}(\Delta)$ means the relative interior of the moment polytope and $\textbf{bar}_{DH}(\Delta)$ is the barycenter of the moment polytope $\Delta(X, K_X^{-1})$ with respect to the Duistermaat--Heckman measure  
$$\prod_{\alpha \in \Phi_{P^u}} \kappa(\alpha, p) \, dp.$$   
Here, $\kappa$ denotes the Killing form on the Lie algebra $\mathfrak g$ of $G$ and $p \in \mathfrak X(T) \otimes \mathbb R$.   
\end{proposition}

By Proposition \ref{Greatest Ricci lower bounds}, we immediately get the following elementary geometric expression for $R(X)$.

\begin{corollary} 
\label{formula for greatest Ricci lower bounds}
Let $X$ be a Fano horospherical manifold without a K\"{a}hler--Einstein metric. 
Denote by $R(X)$ the greatest Ricci lower bound of $X$. 
Let $A$ be the point corresponding to $2 \rho_{P} \in \mathfrak X(T)$ and $C$ the barycenter $\textbf{bar}_{DH}(\Delta)$ of the moment polytope $\Delta(X, K_X^{-1})$ with respect to the Duistermaat--Heckman measure. 
If $Q$ is the point at which the half-line starting from the barycenter $C$ in the direction of $A$ intersects the boundary of the moment polytope $\Delta$, 
then we have
$R(X) = \displaystyle \frac{|\overrightarrow{AQ}|}{|\overrightarrow{CQ}|}.$  
\end{corollary}


\subsection{Projective horospherical manifolds of Picard number one}
\label{Projective horospherical manifolds of Picard number one}

In this subsection, $X$ is a projective horospherical manifold of Picard number one with an open $G$-orbit $G/H$ and $P$ a parabolic subgroup of $G$. 
Then $X$ is Fano since it is of Picard number one and every spherical variety is rational (see \cite[Corollary~2.1.3]{Perrin}).

Let $\{ \alpha_1, \cdots, \alpha_n\}$ be a system of simple roots of $G$ following the standard numbering (e.g. \cite{Humphreys}) and $\varpi_i$ be the $i$th fundamental weight. 
For each $i$, let $P^{\alpha_i}$ denote the maximal parabolic subgroup associated to the simple root $\alpha_i$ and $V(\varpi_i)$ be the irreducible $G$-representation with highest weight $\varpi_i$.

For a highest weight vector $v_i$ of $V(\varpi_{i})$, the $G$-orbit of $[v_i]$ in $\mathbb P( V(\varpi_i))$ is closed and isomorphic to $G/P^{\alpha_i}$, which is a rational homogeneous manifold and denoted by $(G, \alpha_i)$. 
If $v_i$ and $v_j$ are highest weight vectors of $V(\varpi_i)$ and $V(\varpi_j)$ such that $i \neq j$, the open $G$-orbit of $[v_i + v_j] \in \mathbb P(V(\varpi_i) \oplus V(\varpi_j))$ is isomorphic to a $\mathbb C^*$-bundle over a rational homogeneous manifold $G/P$ where $P=P^{\alpha_i}\cap P^{\alpha_j}$. 
Since the closure of the open $G$-orbit is a normal variety according to \cite[Propositon~2.1]{Ho16}, it is a horospherical $G$-variety and we denote it by $(G, \alpha_i, \alpha_j)$.

The projective horospherical manifolds of Picard number one are classified by Pasquier~\cite{Pasquier} using the fact that any nonlinear and $G$-nonhomogeneous ones are of the form $(G, \alpha_i, \alpha_j)$. 
In particular, if $G$-nonhomogeneous $X$ is homogeneous under its automorphism group, it is either projective space, quadric $Q^{2n}$, Grassmannian $\Gr(i,n)$ with $2 \leq i <n$ or spinor variety $\mathbb S_n = \Spin(2n)/P^{\alpha_{n}}$, all which are rational homogeneous manifolds.

Even for nonhomogeneous case, Hong~\cite[Proposition~5.1]{Ho16} proved that a smooth projective horospherical variety $X$ of Picard number one can be embedded as a linear section into a rational homogeneous manifold of Picard number one except when it is $X^1(n)$ for $n \geq 7$. 
However, for $n \geq 7$ it remains open whether $X^1(n)$ is a linear section of a rational homogeneous manifold or not. 
Furthermore, Manivel and Pasquier~\cite{MP20} realized the blow-up of $X$ along the closed orbit $Z$ as the zero locus of a general section of a vector bundle over some homogeneous space. 
For numerical invariants such as dimensions and the first Chern classes of $X$, see \cite[Section~1.6]{GPPS21}.

Each of these nonhomogeneous horospherical variety has many interesting geometry. 
For $n \geq k \geq 2$, $X^3(n, k)$ is isomorphic to the so-called an \emph{odd symplectic Grassmannian} $\text{SGr}(k, 2n+1)$ studied in \cite{Mihai07}, which is defined as a variety parametrizing $k$-dimensional isotropic subspaces in a $(2n+1)$-dimensional complex vector space equipped with a skew-symmetric bilinear form of maximal rank. 
They appears as smooth Schubert varieties in symplectic Grassmannians (\cite{Ho15}) of which standard embeddings are characterized by means of varieties of minimal rational tangents in \cite{KP18}. 
It is also known that odd symplectic Grassmannians are rigid under global K{\"a}hler deformation. 
This is proved in \cite{Park16} for odd Lagrangian Grassmannians $X^3(n,n)$ and in \cite{HL19} for general odd symplectic Grassmannians $X^3(n,k)$. 
Furthermore, Hwang and Li \cite{HL19} characterized odd symplectic Grassmannians $X^3(n,k)$, among Fano manifolds of Picard number one, by their varieties of minimal rational tangents at a general point. 
In addition, Kim~\cite[Section 3]{Kim17} proved the vanishing of the first Lie algebra cohomology of $X$ for their characterizations. 

Recently, Kanemitsu~\cite{Kanemitsu21} proved that the tangent bundles of $X^4$ and $X^1(n)$  for $n \geq 4$ are unstable in the sense of Mumford--Takemoto which disproves a conjecture on stability of tangent bundles of Fano manifolds of Picard number one. 

We know that $X^2$ is isomorphic to a general hyperplane section of the 10-dimensional spinor variety $\mathbb S_5 \subset \mathbb P^{15}$. 
It was studied, for instance, in \cite{Mukai89} and \cite{HM05} as a Fano manifold of coindex 3 and as the variety of minimal rational tangents of a 15-dimensional $F_4$-homogeneous variety $F_4/P^{\alpha_4}$, respectively. 

About $X^5$, there is a result by Pasquier and Perrin \cite{PasquierPerrin} that construct explicitly a smooth flat Fano deformation from the orthogonal Grassmannian $(B_3, \alpha_2) = \text{OGr}(2, 7)$ to $X^5$ using complexified octonions.

\vskip 1em 

Combining Propositions~\ref{spherical data of horospherical space}, \ref{expression of anticanonical divisor}, \ref{moment polytope of spherical variety}, Theorem \ref{horosphercal} and \cite[Theorem 0.2]{Pasquier}, 
we can compute the  moment polytopes of the anticanonical line bundles on nonhomogeneous horospherical manifolds $(G, \alpha_i, \alpha_j)$ of Picard number one. 

\begin{lemma}
\label{moment polytope}
Let $X$ be a  nonhomogeneous projective horospherical manifold of Picard number one. 
If $X$ is of type $(G, \alpha_i, \alpha_j)$, 
then the spherical weight lattice $\mathcal M$ is generated by $\varpi_i-\varpi_j$ and 
the moment polytope $\Delta(X, K_{X}^{-1})$ is the line segment 
$$\{ 2 \rho_P + t (\varpi_i - \varpi_j) : - \langle \alpha_i^{\vee}, 2 \rho_P \rangle \leq t \leq \langle \alpha_j^{\vee}, 2 \rho_P \rangle \}$$
in $2 \rho_P + \mathcal M \otimes \mathbb R$, where $2\rho_P$ is the sum of all roots in the unipotent radical of a submaximal parabolic subgroup $P^{\alpha_i} \cap P^{\alpha_j}$. 
\end{lemma}

\begin{proof}  
By Propositions~\ref{spherical data of horospherical space} and Theorem~\ref{horosphercal}, the spherical weight lattice $\mathcal M$ is generated by $\varpi_i-\varpi_j$.
By \cite[Theorem 0.2]{Pasquier} and Theorem~\ref{horosphercal}, one can see that  $X$ has three $G$-orbits and the codimensions of the two closed orbits are at least two. 
By Proposition~\ref{expression of anticanonical divisor}, $-K_X = \langle \alpha_i^{\vee}, 2 \rho_P \rangle D_i +\langle \alpha_j^{\vee}, 2\rho_P \rangle D_j$ is a $B$-stable divisor representing the anticanonical line bundle $K_X^{-1}$. 
Because $\rho(D_i) = \alpha_i^{\vee}$ and $\rho(D_j) = \alpha_j^{\vee}$, Proposition~\ref{moment polytope of spherical variety} gives us the inequalities 
$$ \left\langle \frac{\alpha_i^{\vee}}{\langle \alpha_i^{\vee}, 2 \rho_P \rangle}, \, t(\varpi_i - \varpi_j) \right\rangle \geq -1 \iff  t \geq - \langle \alpha_i^{\vee}, 2 \rho_P \rangle$$
and  
$$\left\langle \frac{\alpha_j^{\vee}}{\langle \alpha_j^{\vee}, 2\rho_P \rangle}, \, t(\varpi_i - \varpi_j) \right\rangle \geq -1 \iff t \leq \langle \alpha_j^{\vee}, 2\rho_P \rangle. $$
Therefore, the result follows.
\end{proof}


\section{Greatest Ricci lower bounds of \\ projective horospherical manifolds of Picard number one}
\label{Greatest Ricci lower bounds of smooth projective horospherical varieties of Picard number one}

In this section we prove Theorem~\ref{Main theorem}. 
In the course of the proof, we explicitly present the algebraic moment polytope using Lemma \ref{moment polytope} 
and then compute the greatest Ricci lower bound for every nonhomogeneous projective horospherical manifold of Picard number one 
using Corollary~\ref{formula for greatest Ricci lower bounds}.  
The most tricky part of the proof is to compute the barycenter of the moment polytope with respect to the Duistermaat--Heckman measure.

\subsection{Greatest Ricci lower bound of $X^5$} 

Let $X^5$ be the projective horospherical manifold of type $(G_2, \alpha_2, \alpha_1)$ in Theorem \ref{horosphercal}.
The spherical weight lattice $\mathcal M$ is generated by $\varpi_1 - \varpi_2$ for the fundamental weights $\varpi_1, \varpi_2 \in \mathfrak X(T)$ of a maximal torus $T \subset G_2$, 
so that the dual lattice $\mathcal N$ is generated by $\frac{1}{2} \alpha_1^{\vee} - \frac{1}{2} \alpha_2^{\vee}$ for the coroots $\alpha_1^{\vee}$ and $\alpha_2^{\vee}$. 

Choose a realization of the root system $G_2$ in the Euclidean plane $\mathbb R^2$ with  $\alpha_1 = (1, 0)$ and $\alpha_2 = \left(- \frac{3}{2}, \frac{\sqrt{3}}{2} \right)$.
Then, the complex Lie group $G_2$ has the six positive roots:
\begin{align*}
\Phi^+ &= \left\{ \alpha_1, \alpha_2, \alpha_1 + \alpha_2, 2 \alpha_1 + \alpha_2, 3\alpha_1 + \alpha_2, 3\alpha_1 + 2\alpha_2 \right\}
\\
& = \left\{ (1, 0), \left(-\frac{3}{2}, \frac{\sqrt{3}}{2}\right), \left(-\frac{1}{2}, \frac{\sqrt{3}}{2}\right), \left(\frac{1}{2}, \frac{\sqrt{3}}{2}\right), \left(\frac{3}{2}, \frac{\sqrt{3}}{2}\right), (0, \sqrt{3}) \right\}.
\end{align*}
From the relation $(\alpha_i^{\vee}, \varpi_j)=\delta_{ij}$, 
the fundamental weights corresponding to the system of simple roots are 
$\varpi_1 = \left(\frac{1}{2}, \frac{\sqrt{3}}{2} \right)$ and $\varpi_2 = (0, \sqrt{3})$. 

\begin{figure}
\begin{minipage}[b]{\textwidth} 
\centering
\begin{tikzpicture}
\clip (-2.2,-0.5) rectangle (4.5,7.5); 
\begin{scope}[y=(60:1)]

\coordinate (pi1) at (0,1);
\coordinate (pi2) at (-1,2);
\coordinate (v1) at ($4*(pi1)$);
\coordinate (v2) at ($4*(pi2)$);
\coordinate (a1) at (1,0);
\coordinate (a2) at (-2,1);
\coordinate (a3) at ($3*(a1)+(a2)$);
\coordinate (barycenter) at (-67/28, 179/28);

\coordinate (Origin) at (0,0);
\coordinate (asum) at ($(a1)+(a2)$);
\coordinate (2rho) at (-2,6);

\foreach \x  in {-8,-7,...,9}{
  \draw[help lines,dashed]
    (\x,-8) -- (\x,9)
    (-8,\x) -- (9,\x) 
    [rotate=60] (\x,-8) -- (\x,9) ;
}

\fill (Origin) circle (2pt) node[below left] {0};
\fill (pi1) circle (2pt) node[right] {$\varpi_1$};
\fill (pi2) circle (2pt) node[right] {$\varpi_2$};
\fill (a1) circle (2pt) node[below] {$\alpha_1$};
\fill (a2) circle (2pt) node[above left] {$\alpha_2$};
\fill (a3) circle (2pt) node[below right] {$3\alpha_1+\alpha_2$};

\fill (asum) circle (2pt) node[above] {$\alpha_1+\alpha_2 \quad$};
\fill (2rho) circle (2pt) node[right] {$\, 2\rho$};
\fill (v2) circle (2pt) node[below right] {$\quad \Delta(X^5, K^{-1}_{X^5})$};

\fill (v1) circle (2pt) node[below right] {$4\varpi_1$};
\fill (v2) circle (2pt) node[above left] {$4\varpi_2$};

\fill (barycenter) circle (2pt) node[blue, above right] {$\textbf{bar}_{DH}(\Delta_5)$};

\draw[->,,thick](Origin)--(pi1);
\draw[->,,thick](Origin)--(pi2); 
\draw[->,,thick](Origin)--(a1);
\draw[->,,thick](Origin)--(a2);
\draw[->,,thick](Origin)--(a3); 
\draw[->,,thick](Origin)--(asum); 
\draw[thick,gray](v1)--(v2);

\draw [shorten >=-8cm, blue, thick, dashed] (Origin) to ($(pi1)$);
\draw [shorten >=-6cm, blue, thick, dashed] (Origin) to ($(pi2)$);

\end{scope}
\end{tikzpicture} 

\caption{Moment polytope $\Delta_5=\Delta(X^5, K^{-1}_{X^5})$.}
\label{Delta_5}
\end{minipage}
\end{figure}
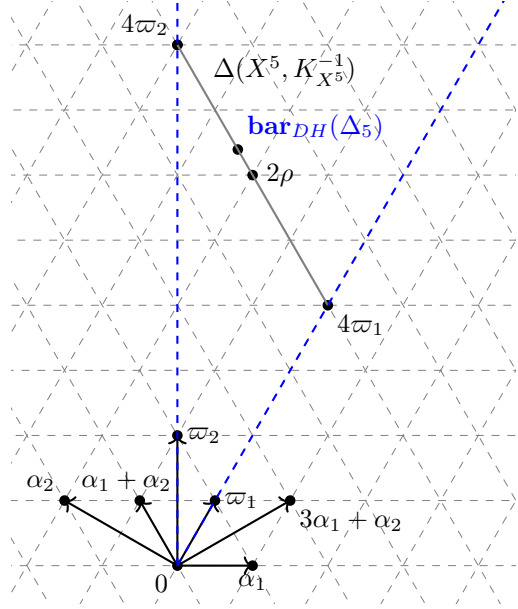

\begin{proposition} 
\label{moment polytope_5}
The moment polytope $\Delta_5 = \Delta(X^5, K_{X^5}^{-1})$ is the line segment connecting two points $4 \varpi_1$ and $4 \varpi_2$ in $\mathfrak X(T) \otimes \mathbb R$, that is, 
\[
\Delta(X^5, K_{X^5}^{-1}) = \{ (2 \varpi_1 + 2 \varpi_2) + t(\varpi_1 - \varpi_2) : -2 \leq t \leq 2 \}.
\] 
\end{proposition}

\begin{proof}
Since $P = P^{\alpha_1} \cap P^{\alpha_2} = B$, we have $2 \rho_P = 2 \rho_{G_2} = 2 \varpi_1 + 2 \varpi_2$ 
and the assertion follows from Lemma~\ref{moment polytope}
(see Figure~\ref{Delta_5}). 
\end{proof}

\begin{remark} 
Considering the minimal embedding 
$X^5 \hookrightarrow \mathbb P(V_{G_2}(\varpi_1) \oplus V_{G_2}(\varpi_2) )$ and $K_{X^5}^{-1} = \mathcal O(4)$, motivated by the original definition of Brion, 
we have 
\begin{align*}
H^0(X^5, \mathcal O(1)) & = V_{G_2}(\varpi_1) \oplus V_{G_2}(\varpi_2) \quad \text{and} \\
H^0(X^5, K_{X^5}^{-1}) & = V_{G_2}(4 \varpi_1) \oplus V_{G_2}(3 \varpi_1 + \varpi_2) \oplus V_{G_2}(2 \varpi_1 + 2 \varpi_2) \oplus V_{G_2}(\varpi_1 + 3 \varpi_2) \oplus V_{G_2}(4\varpi_2). 
\end{align*}
\end{remark} 

\begin{proposition} 
\label{GRLB_X5} 
The greatest Ricci lower bound of $X^5$ is $R(X^5) = \displaystyle \frac{56}{67} \approx 0.8358$. 
\end{proposition}

\begin{proof}
For $p=(x, y)$, the Duistermaat--Heckman measure on $\mathfrak X(T) \otimes \mathbb R$ is given as   
\[
\prod_{\alpha \in \Phi^+} \kappa(\alpha, p) \, dp 
= x \Big(-\frac{3}{2}x + \frac{\sqrt{3}}{2}y\Big) \Big(-\frac{1}{2}x + \frac{\sqrt{3}}{2}y\Big) \Big(\frac{1}{2}x + \frac{\sqrt{3}}{2}y\Big) \Big(\frac{3}{2}x + \frac{\sqrt{3}}{2}y\Big) (\sqrt{3}y) \, dxdy =: P_{DH}(x, y) \, dxdy.
\]
From Proposition~\ref{moment polytope_5}, 
the moment polytope $\Delta_5$ is parametrized as 
$$\gamma_5(t) = \left(1+\frac{t}{2}, \frac{\sqrt{3}}{2}(2+t) + \sqrt{3}(2-t) \right)$$ 
for $-2 \leq t \leq 2$. 
Then, we can compute the volume (or length) 
\[
\text{Vol}_{DH}(\Delta_5) = \displaystyle \int_{-2}^{2} P_{DH}(\gamma_5(t)) \, dt = 9216,  
\]
and the barycenter 
$$\textbf{bar}_{DH}(\Delta_5) = (\bar{x}, \bar{y}) 
= \frac{1}{\text{Vol}_{DH}(\Delta_5)} \displaystyle \int_{\Delta_5}p \prod_{\alpha \in \Phi^+} \kappa(\alpha, p) \, dp 
= \left(\frac{45}{56}, \frac{179 \sqrt{3}}{56} \right) \approx (0.804, 3.196 \times \sqrt{3})$$ 
of the moment polytope $\Delta_5$ with respect to the Duistermaat--Heckman measure. 
Since $\textbf{bar}_{DH}(\Delta_5) = \gamma_5(-\frac{11}{28})$, 
we get the greatest Ricci lower bound $R(X^5) = \displaystyle \frac{2}{2+\frac{11}{28}} = \frac{56}{67}$ by Corollary~\ref{formula for greatest Ricci lower bounds}.
\end{proof}

Using the basis of the weight lattice instead of the orthogonal coordinate system is more effective for calculations of the remaining cases. 
From the relations  
\[
\kappa (\alpha_1, \varpi_1) = \frac{1}{2}, \, \kappa (\alpha_2, \varpi_1) = 0, \quad 
\kappa (\alpha_1, \varpi_2) = 0, \, \kappa (\alpha_2, \varpi_2) = \frac{3}{2},
\]
we have the Duistermaat--Heckman polynomial evaluated on the moment polytope $\Delta_5$ 
\begin{align*}
P_{DH}(\gamma_5(t)) & = \prod_{\alpha \in \Phi^+} \kappa(\alpha, (2+t)\varpi_1 + (2-t) \varpi_2) \\
& = \frac{1}{2}(2+t) \cdot \frac{3}{2}(2-t) \cdot \left \{ \frac{1}{2}(2+t) + \frac{3}{2}(2-t) \right \} \cdot \left \{ 2 \cdot \frac{1}{2} (2+t) + \frac{3}{2}(2-t) \right \} \\
& \times \left \{ 3 \cdot \frac{1}{2}(2+t) + \frac{3}{2}(2-t) \right \} \cdot \left \{ 3 \cdot \frac{1}{2}(2+t) + 2 \cdot \frac{3}{2}(2-t) \right \}. 
\end{align*}
As $\displaystyle \frac{1}{\text{Vol}_{DH}(\Delta_5)} \int_{-2}^{2} (2+t) \cdot P_{DH}(\gamma_5(t)) \, dt = \frac{45}{28}$ and 
$\displaystyle \frac{1}{\text{Vol}_{DH}(\Delta_5)} \int_{-2}^{2} (2-t) \cdot P_{DH}(\gamma_5(t)) \, dt = \frac{67}{28}$, 
we get the same result $\displaystyle \textbf{bar}_{DH}(\Delta_5) = \frac{45}{28} \varpi_1 + \frac{67}{28} \varpi_2 = \gamma_5 \left(-\frac{11}{28}\right)$.

\subsection{Greatest Ricci lower bound of $X^2$} 
Let $X^2$ be the projective horospherical manifold of type $(B_3, \alpha_1, \alpha_3)$ in Theorem~\ref{horosphercal}. 

\begin{proposition} 
\label{moment polytope_2}
The moment polytope $\Delta_2 = \Delta(X^2, K_{X^2}^{-1})$ is the line segment connecting two points $7 \varpi_1$ and $7 \varpi_3$ in $\mathfrak X(T) \otimes \mathbb R$, that is, 
\[
\Delta(X^2, K_{X^2}^{-1}) = \{ (3\varpi_1 + 4\varpi_3) + t(\varpi_1 - \varpi_3) : -3 \leq t \leq 4 \}.
\] 
\end{proposition}

\begin{proof}
Since the stabilizer of the open $B$-orbit in $X^2$ is $P=P^{\alpha_1}\cap P^{\alpha_3}$, we know that $\Phi_{P^u} = \Phi^+_{B_3} \setminus \{ \alpha_2 \}$. 
Recall that the fundamental weights $\varpi_1, \varpi_2, \varpi_3$ and the second simple root $\alpha_2$ of $B_3$ can be expressed as 
$\varpi_1 = L_1, \varpi_2 = L_1+L_2, \varpi_3 = \frac{L_1+L_2+L_3}{2}$ and $\alpha_2 = L_2 - L_3$ in terms of an orthonormal basis $\{ L_1, L_2, L_3 \}$ of the dual Cartan subalgebra.
Then, as $\alpha_2 = -\varpi_1 + 2\varpi_2 - 2\varpi_3$ we have 
$$2 \rho_P = 2\rho_{B_3} - 2\rho_L = 2(\varpi_1 + \varpi_2 + \varpi_3) - \alpha_2 = 2(\varpi_1 + \varpi_2 + \varpi_3) - (-\varpi_1 + 2\varpi_2 - 2\varpi_3) = 3\varpi_1 + 4\varpi_3.$$ 
The result follows from Lemma~\ref{moment polytope}. 
\end{proof} 

\begin{proposition} 
\label{GRLB_X2} 
The greatest Ricci lower bound of $X^2$ is $R(X^2) = \displaystyle \frac{20}{21} \approx 0.952$. 
\end{proposition}

\begin{proof}
By Proposition~\ref{moment polytope_2}, 
the moment polytope $\Delta_2$ is parametrized as 
$$\gamma_2(t) = (3+t)\varpi_1 + (4-t) \varpi_3$$ 
for $-3 \leq t \leq 4$. 
From the relations  
\[
\kappa (\alpha_1, \varpi_1) = 1, \, \kappa (\alpha_2, \varpi_1) = 0, \, \kappa (\alpha_3, \varpi_1) = 0, \quad 
\kappa (\alpha_1, \varpi_3) = 0, \, \kappa (\alpha_2, \varpi_3) = 0, \, \kappa (\alpha_3, \varpi_3) = \frac{1}{2}
\]
and $\Phi_{P^u} = \Phi^+_{B_3} \setminus \{ \alpha_2 \} = \{ \alpha_1, \alpha_3, \alpha_1 + \alpha_2, \alpha_2 + \alpha_3, \alpha_1 + \alpha_2 + \alpha_3, \alpha_2 + 2 \alpha_3, \alpha_1 + \alpha_2 + 2 \alpha_3, \alpha_1 + 2 \alpha_2 + 2 \alpha_3 \}$,  
we have the Duistermaat--Heckman polynomial evaluated on the moment polytope $\Delta_2$ 
\begin{align*}
P_{DH}(\gamma_2(t)) & = \prod_{\alpha \in \Phi_{P^u}} \kappa(\alpha, (3+t) \varpi_1 + (4-t) \varpi_3) \\
& = (3+t) \cdot \frac{1}{2}(4-t) \cdot (3+t) \cdot \frac{1}{2}(4-t) \cdot \left \{ (3+t) + \frac{1}{2}(4-t) \right \} \cdot \left \{ 2 \cdot \frac{1}{2}(4-t) \right \} \\
& \times \left \{ (3+t) + 2 \cdot \frac{1}{2}(4-t) \right \} \cdot \left \{ (3+t) + 2 \cdot \frac{1}{2}(4-t) \right \} \\
& = \frac{49}{4}(3+t)^2 (4-t)^3 \left (5 + \frac{t}{2} \right ). 
\end{align*}
As $\displaystyle \frac{1}{\text{Vol}_{DH}(\Delta_2)} \int_{-3}^{4} (3+t) \cdot P_{DH}(\gamma_2(t)) \, dt = \frac{63}{20}$ and 
$\displaystyle \frac{1}{\text{Vol}_{DH}(\Delta_2)} \int_{-3}^{4} (4-t) \cdot P_{DH}(\gamma_2(t)) \, dt = \frac{77}{20}$, 
we get the barycenter 
$\displaystyle \textbf{bar}_{DH}(\Delta_2) = \frac{63}{20} \varpi_1 + \frac{77}{20} \varpi_3 = \gamma_2 \left(\frac{3}{20}\right)$. 
Therefore, the greatest Ricci lower bound of $X^2$ is $R(X^2) = \displaystyle \frac{3}{3+\frac{3}{20}} = \frac{20}{21}$ by Corollary~\ref{formula for greatest Ricci lower bounds}.
\end{proof}

\subsection{Greatest Ricci lower bound of $X^4$} 
Let $X^4$ be the projective horospherical manifold of type $(F_4, \alpha_2, \alpha_3)$  in Theorem~\ref{horosphercal}. 

\begin{proposition} 
\label{moment polytope_4}
The moment polytope $\Delta_4 = \Delta(X^4, K_{X^4}^{-1})$ is the line segment connecting two points $6 \varpi_2$ and $6 \varpi_3$ in $\mathfrak X(T) \otimes \mathbb R$, that is, 
\[
\Delta(X^4, K_{X^4}^{-1}) = \{ (3\varpi_2 + 3\varpi_3) + t(\varpi_2 - \varpi_3) : -3 \leq t \leq 3 \}.
\] 
\end{proposition}

\begin{proof}
Since $P = P^{\alpha_2} \cap P^{\alpha_3}$, we know that $\Phi_{P^u} = \Phi^+_{F_4} \setminus \{ \alpha_1, \alpha_4 \}$ 
and $2 \rho_P = 2\rho_{F_4} - \alpha_1 - \alpha_4 = 2(\varpi_1 + \varpi_2 + \varpi_3 + \varpi_4) - (2 \varpi_1 - \varpi_2 - \varpi_3 + 2\varpi_4) = 3\varpi_2 + 3\varpi_3$. 
The result follows from Lemma~\ref{moment polytope}. 
\end{proof}

\begin{proposition} 
\label{GRLB_X4} 
The greatest Ricci lower bound of $X^4$ is $R(X^4) = \displaystyle \frac{178992099}{243545402} \approx 0.734$. 
\end{proposition}

\begin{proof}
By Proposition~\ref{moment polytope_4}, 
the moment polytope $\Delta_4$ is parametrized as 
$$\gamma_4(t) = (3+t)\varpi_2 + (3-t) \varpi_3$$ 
for $-3 \leq t \leq 3$. 
From the relations  
\begin{align*}
\kappa (\alpha_1, \varpi_2) = 0, \, \kappa (\alpha_2, \varpi_2) = 1, \, \kappa (\alpha_3, \varpi_2) = 0, \, \kappa (\alpha_4, \varpi_2) = 0, \\ 
\kappa (\alpha_1, \varpi_3) = 0, \, \kappa (\alpha_2, \varpi_3) = 0, \, \kappa (\alpha_3, \varpi_3) = \frac{1}{2}, \, \kappa (\alpha_4, \varpi_3) = 0
\end{align*}
and $\Phi_{P^u} = \Phi^+_{F_4} \setminus \{ \alpha_1, \alpha_4 \} 
= \{ \alpha_2, \alpha_3, \alpha_1 + \alpha_2, \alpha_2 + \alpha_3, \alpha_3 + \alpha_4, \alpha_1 + \alpha_2 + \alpha_3, \alpha_2 + 2 \alpha_3, 
\alpha_2 + \alpha_3 + \alpha_4, \alpha_1 + \alpha_2 + 2 \alpha_3, \alpha_1 + \alpha_2 + \alpha_3 + \alpha_4, \alpha_2 + 2\alpha_3 + \alpha_4, 
\alpha_1 + 2\alpha_2 + 2 \alpha_3, \alpha_1 + \alpha_2 + 2 \alpha_3 + \alpha_4, \alpha_2 + 2 \alpha_3 + 2 \alpha_4, \alpha_1 + 2 \alpha_2 + 2 \alpha_3 + \alpha_4, \alpha_1 + \alpha_2 + 2 \alpha_3 + 2 \alpha_4, \alpha_1 + 2 \alpha_2 + 3 \alpha_3 + \alpha_4, \alpha_1 + 2 \alpha_2 + 2 \alpha_3 + 2 \alpha_4, \alpha_1 + 2 \alpha_2 + 3 \alpha_3 + 2 \alpha_4, \alpha_1 + 2 \alpha_2 + 4 \alpha_3 + 2 \alpha_4, \alpha_1 + 3 \alpha_2 + 4 \alpha_3 + 2 \alpha_4, 2 \alpha_1 + 3 \alpha_2 + 4 \alpha_3 + 2 \alpha_4 \}$, 
we have the Duistermaat--Heckman polynomial evaluated on the moment polytope $\Delta_4$ 
\begin{align*}
P_{DH}(\gamma_4(t)) & = \prod_{\alpha \in \Phi_{P^u}} \kappa(\alpha, (3+t) \varpi_2 + (3-t) \varpi_3) \\
& = (3+t)^2 \left \{ \frac{1}{2}(3-t) \right \}^2 \left \{ (3+t) + \frac{1}{2}(3-t) \right \}^4 \left \{ (3+t) + 2 \cdot \frac{1}{2}(3-t) \right \}^6 \left \{2(3+t) + 2 \cdot \frac{1}{2}(3-t) \right \}^3 \\
& \times \left \{2(3+t) + 3 \cdot \frac{1}{2}(3-t) \right \}^2 \left \{2(3+t) + 4 \cdot \frac{1}{2}(3-t) \right \} \left \{3(3+t) + 4 \cdot \frac{1}{2}(3-t) \right \}^2. 
\end{align*}
As $\displaystyle \frac{1}{\text{Vol}_{DH}(\Delta_4)} \int_{-3}^{3} (3+t) \cdot P_{DH}(\gamma_4(t)) \, dt = \frac{243545402}{59664033}$ and 
$\displaystyle \frac{1}{\text{Vol}_{DH}(\Delta_4)} \int_{-3}^{3} (3-t) \cdot P_{DH}(\gamma_4(t)) \, dt = \frac{114438796}{59664033}$, 
we get the barycenter 
$\displaystyle \textbf{bar}_{DH}(\Delta_4) = \frac{243545402}{59664033} \varpi_2 + \frac{114438796}{59664033} \varpi_3 = \gamma_4 \left(\frac{64553303}{59664033} \right)$. 
Therefore, the greatest Ricci lower bound of $X^4$ is $R(X^4) = \displaystyle \frac{3}{3+\frac{64553303}{59664033}} = \frac{178992099}{243545402}$ by Corollary~\ref{formula for greatest Ricci lower bounds}.
\end{proof}

\subsection{Greatest Ricci lower bounds of $X^1(n)$} 
For $n \geq 3$, let $X^1(n)$ be the projective horospherical manifold of type $(B_n, \alpha_{n-1}, \alpha_n)$  in Theorem~\ref{horosphercal}.
 
\begin{proposition} 
\label{moment polytope_1}
The moment polytope $\Delta_1 = \Delta(X^1(n), K_{X^1(n)}^{-1})$ is the line segment connecting two points $(n+2) \varpi_{n-1}$ and $(n+2) \varpi_n$ in $\mathfrak X(T) \otimes \mathbb R$, that is, 
\[
\Delta(X^1(n), K_{X^1(n)}^{-1}) = \{ (n \varpi_{n-1} + 2 \varpi_n) + t(\varpi_{n-1} - \varpi_n) : -n \leq t \leq 2 \}.
\] 
\end{proposition}

\begin{proof}
Let $L_1, L_2, \cdots, L_n$ be an orthonormal basis of $\mathfrak X(T) \otimes \mathbb R$. 
Then the set of all positive roots of $B_n$ consists of the following: for $1 \leq i < j \leq n$,
$$L_i-L_j, \quad L_i+L_j, \quad L_i.$$
Hence, $2\rho_G = \sum_{i<j}(L_i-L_j) + \sum_{i<j}(L_i+L_j) + \sum_{i=1}^{n}L_i = \sum_{i=1}^{n}(2n-2i+1)L_i$.
On the other hand, since the reductive part $L$ of $P = P^{\alpha_{n-1}} \cap P^{\alpha_n}$ is of type $A_{n-2}$, 
the positive roots of $L$ are $L_i-L_j$ for $1 \leq i<j \leq n-1$ and $2\rho_L = \sum_{1 \leq i<j \leq n-1}(L_i-L_j) = \sum_{i=1}^{n-1} (n-2i) L_i$. 
Therefore, $$2\rho_G - 2\rho_L = \sum_{i=1}^{n} (2n-2i+1) L_i - \sum_{i=1}^{n-1} (n-2i) L_i = n \sum_{i=1}^{n-1} L_i + \sum_{i=1}^{n}L_i = n \varpi_{n-1} + 2 \varpi_{n}.$$ 
The result follows from Lemma~\ref{moment polytope}.
\end{proof}

\begin{proposition} 
\label{GRLB_X1} 
The greatest Ricci lower bound of $X^1(n)$ is 
$$R(X^1(n)) = \displaystyle \frac{n \displaystyle \int_{-n}^{2} (2-t) (n+t)^{n-1} (t+2n+2)^{\frac{n(n-1)}{2}} \, dt}{\displaystyle \int_{-n}^{2} (2-t) (n+t)^{n} (t+2n+2)^{\frac{n(n-1)}{2}} \, dt}.$$ 
\end{proposition}

\begin{proof}
By Proposition~\ref{moment polytope_1}, 
the moment polytope $\Delta_1$ is parametrized as 
$$\gamma_1(t) = (n+t)\varpi_{n-1} + (2-t) \varpi_n$$ 
for $-n \leq t \leq 2$. 
In order to obtain the Duistermaat--Heckman polynomial evaluated on the moment polytope $\Delta_1$, 
the number of roots in $\Phi_{P^u}$ for each value of the coefficients of $\alpha_{n-1}, \alpha_n$ must be counted specifically:  
\begin{enumerate}
\item[(i)] the cardinality of the set $\{ \sum_{i=1}^n c_i \alpha_i \in \Phi^+ : c_{n-1} = 1, c_n = 0 \}$ is $n-1$, 
\item[(ii)] the cardinality of the set $\{ \sum_{i=1}^n c_i \alpha_i \in \Phi^+ : c_{n-1} = 0, c_n = 1 \}$ is $1$, 
\item[(iii)] the cardinality of the set $\{ \sum_{i=1}^n c_i \alpha_i \in \Phi^+ : c_{n-1} = 1, c_n = 1 \}$ is $n-1$, 
\item[(iv)] the cardinality of the set $\{ \sum_{i=1}^n c_i \alpha_i \in \Phi^+ : c_{n-1} = 1, c_n = 2 \}$ is $n-1$, and 
\item[(v)] the cardinality of the set $\{ \sum_{i=1}^n c_i \alpha_i \in \Phi^+ : c_{n-1} = 2, c_n = 2 \}$ is $\frac{(n-1)(n-2)}{2}$. 
\end{enumerate}
Then, from $\kappa (\alpha_{n-1}, \varpi_{n-1}) = 1$ and $\kappa (\alpha_n, \varpi_n) = \frac{1}{2}$ we have 
\begin{align*}
P_{DH}(\gamma_1(t)) & = \prod_{\alpha \in \Phi_{P^u}} \kappa(\alpha, (n+t) \varpi_{n-1} + (2-t) \varpi_n) \\
& = (n+t)^{n-1} \cdot \frac{1}{2}(2-t) \left \{ (n+t) + \frac{1}{2}(2-t) \right \}^{n-1} \\
& \times \left \{ (n+t) + 2 \cdot \frac{1}{2}(2-t) \right \}^{n-1} \left \{ 2(n+t) + 2 \cdot \frac{1}{2}(2-t) \right \}^{\frac{(n-1)(n-2)}{2}} \\
& = (n+t)^{n-1} \cdot \frac{1}{2} (2-t) \left ( \frac{1}{2} t + n+1 \right )^{n-1}(n+2)^{n-1} (t+2n+2)^{\frac{(n-1)(n-2)}{2}} \\
& = \frac{(n+2)^{n-1}}{2^n} (2-t) (n+t)^{n-1} (t+2n+2)^{\frac{n(n-1)}{2}}. 
\end{align*}
As $\displaystyle \frac{1}{\text{Vol}_{DH}(\Delta_1)} \int_{-n}^{2} (2-t) \cdot P_{DH}(\gamma_1(t)) \, dt = \frac{\int_{-n}^{2} (2-t)^2 (n+t)^{n-1} (t+2n+2)^{\frac{n(n-1)}{2}} \, dt }{\int_{-n}^{2} (2-t) (n+t)^{n-1} (t+2n+2)^{\frac{n(n-1)}{2}} \, dt }$, 
we get the barycenter 
$\displaystyle \textbf{bar}_{DH}(\Delta_1) = \gamma_1 \left( 2-\frac{\int_{-n}^{2} (2-t)^2 (n+t)^{n-1} (t+2n+2)^{\frac{n(n-1)}{2}} \, dt }{\int_{-n}^{2} (2-t) (n+t)^{n-1} (t+2n+2)^{\frac{n(n-1)}{2}} \, dt } \right)$
of the moment polytope $\Delta_1$ with respect to the Duistermaat--Heckman measure.  
Since $2 - \displaystyle \frac{\int_{-n}^{2} (2-t)^2 (n+t)^{n-1} (t+2n+2)^{\frac{n(n-1)}{2}} \, dt }{\int_{-n}^{2} (2-t) (n+t)^{n-1} (t+2n+2)^{\frac{n(n-1)}{2}} \, dt } > 0$
by Lemma~\ref{inequalty:X1}, 
the greatest Ricci lower bound of $X^1(n)$ is 
\begin{align*}
R(X^1(n)) & = \displaystyle \frac{n}{\left (2 - \frac{\int_{-n}^{2} (2-t)^2 (n+t)^{n-1} (t+2n+2)^{\frac{n(n-1)}{2}} \, dt }{\int_{-n}^{2} (2-t) (n+t)^{n-1} (t+2n+2)^{\frac{n(n-1)}{2}} \, dt } \right ) - (-n)} \\
& = \displaystyle \frac{n \int_{-n}^{2} (2-t) (n+t)^{n-1} (t+2n+2)^{\frac{n(n-1)}{2}} \, dt}{(n+2) \int_{-n}^{2} (2-t) (n+t)^{n-1} (t+2n+2)^{\frac{n(n-1)}{2}} \, dt - \int_{-n}^{2} (2-t)^2 (n+t)^{n-1} (t+2n+2)^{\frac{n(n-1)}{2}} \, dt}\\
& = \displaystyle \frac{n \int_{-n}^{2} (2-t) (n+t)^{n-1} (t+2n+2)^{\frac{n(n-1)}{2}} \, dt}{\int_{-n}^{2} (2-t) (n+t)^{n} (t+2n+2)^{\frac{n(n-1)}{2}} \, dt}
\end{align*}
by Corollary~\ref{formula for greatest Ricci lower bounds}.
\end{proof}

\begin{lemma} 
\label{inequalty:X1} 
For any $n \geq 3$, the inequality 
$$\displaystyle \int_{-n}^{2} (2-t)^2 (n+t)^{n-1} (t+2n+2)^{\frac{n(n-1)}{2}} \, dt  < 2 \int_{-n}^{2} (2-t) (n+t)^{n-1} (t+2n+2)^{\frac{n(n-1)}{2}} \, dt$$ holds. 
\end{lemma}

\begin{proof}
It suffices to show that 
$\displaystyle \int_{-n}^{2} t (2-t) (n+t)^{n-1} (t+2n+2)^{\frac{n(n-1)}{2}} \, dt > 0$ for $n \geq 3$. 
For a fixed $n$, since $t+2n+2$ is an increasing function, $t (2-t) (n+t)^{n-1} (t+2n+2)^{\frac{n(n-1)}{2}} > t (2-t) (n+t)^{n-1} (2n+2)^{\frac{n(n-1)}{2}}$ for $0 < t < 2$. 
On the other hand, when $-n < t < 0$ we also have the same inequality 
\newline 
$t (2-t) (n+t)^{n-1} (t+2n+2)^{\frac{n(n-1)}{2}} > t (2-t) (n+t)^{n-1} (2n+2)^{\frac{n(n-1)}{2}}$ since $t<0$. 
Thus, 
\begin{align*}
& \int_{-n}^{2} t (2-t) (n+t)^{n-1} (t+2n+2)^{\frac{n(n-1)}{2}} \, dt \\
& = \int_{-n}^{0} t (2-t) (n+t)^{n-1} (t+2n+2)^{\frac{n(n-1)}{2}} \, dt + \int_{0}^{2} t (2-t) (n+t)^{n-1} (t+2n+2)^{\frac{n(n-1)}{2}} \, dt  \\
& > \int_{-n}^{0} t (2-t) (n+t)^{n-1} (2n+2)^{\frac{n(n-1)}{2}} \, dt + \int_{0}^{2} t (2-t) (n+t)^{n-1} (2n+2)^{\frac{n(n-1)}{2}} \, dt  \\
& = (2n+2)^{\frac{n(n-1)}{2}} \int_{-n}^{2} t (2-t) (n+t)^{n-1} \, dt \\
& = (2n+2)^{\frac{n(n-1)}{2}} \int_{0}^{n+2} (s-n) (n+2-s) s^{n-1} \, ds \qquad \text{ (let $s:= n+t$) } \\
& = (2n+2)^{\frac{n(n-1)}{2}} \int_{0}^{n+2} \{-s^{n+1} + 2(n+1)s^{n} - n(n+2)s^{n-1} \} \, ds \\
& = (2n+2)^{\frac{n(n-1)}{2}} \Big [ - \frac{s^{n+2}}{n+2} + 2(n+1) \cdot \frac{s^{n+1}}{n+1} - n(n+2) \cdot \frac{s^{n}}{n} \Big ]_{0}^{n+2} = 0 
\end{align*}
as claimed. 
The proof is complete. 
\end{proof}

The approximate values of the greatest Ricci lower bounds $R(X^1(n))$ for some $n$ are summarized as shown in the following (the first row of Table \ref{table2}). 

\begin{center}
\begin{tabular}{c c c c c c c c c c c}
		\toprule
		$n$ & $3$   & $4$ & $5$ & $6$ & $7$ & $10$ & $20$ & $30$ & $50$ & $70$
\\
		\midrule 
	$R(X^1(n))$	&	$0.8955$	&	$0.8755$  &	$0.8686$ 	&	$0.8685$ 	&	$0.8715$	& 	$0.8863$  &	$0.9251$ 	&	$0.9451$ 	&	$0.9644$ 	&	$0.9737$
\\
		\bottomrule
\end{tabular}
\end{center}

\vskip 1em 

Interestingly, we observe that the value $R(X^1(n))$ decreases from $n=3$ to $6$ and then increases again from $n=7$.
In addition, $R(X^1(n))$ becomes arbitrarily close to $1$ as $n$ becomes larger and larger. 

\begin{corollary} 
\label{limit of R(X1)}
The greatest Ricci lower bound $R(X^1(n))$ of $X^1(n)$ converges to $1$ as $n$ increases, that is, $\displaystyle \lim_{n \to \infty} R(X^1(n)) = 1$. 
\end{corollary}

\begin{proof}
From Proposition~\ref{GRLB_X1}, we have 
$R(X^1(n)) = \displaystyle \frac{n}{n +2 - \frac{\int_{-n}^{2} (2-t)^2 (n+t)^{n-1} (t+2n+2)^{\frac{n(n-1)}{2}} \, dt }{\int_{-n}^{2} (2-t) (n+t)^{n-1} (t+2n+2)^{\frac{n(n-1)}{2}} \, dt } }$. 
Since all functions $2-t$, $n+t$ and $t+2n+2$ in the integrands are positive for $-n < t < 2$, 
we get an inequality $R(X^1(n)) > \displaystyle \frac{n}{n+2}$; 
hence $\displaystyle \lim_{n \to \infty} R(X^1(n)) \geq \lim_{n \to \infty} \frac{n}{n+2} = 1$. 
As $R(X^1(n)) \leq 1$ by definition, we conclude that $\displaystyle \lim_{n \to \infty} R(X^1(n)) = 1$. 
\end{proof}

\subsection{Greatest Ricci lower bounds of $X^3(n, k)$} 

For $n \geq k \geq 2$, let $X^3(n, k)$ be the projective horospherical manifold of type $(C_n, \alpha_{k}, \alpha_{k-1})$ in Theorem~\ref{horosphercal}.
 
\begin{proposition} 
\label{moment polytope_3}
The moment polytope $\Delta_3 = \Delta(X^3(n,k), K_{X^3(n,k)}^{-1})$ is the line segment connecting two points $(2n-k+2) \varpi_{k-1}$ and $(2n-k+2) \varpi_k$ in $\mathfrak X(T) \otimes \mathbb R$, that is, 
\[
\Delta(X^3(n,k), K_{X^3(n,k)}^{-1}) = \{ (k \varpi_{k-1} + (2n-2k+2) \varpi_k) + t(\varpi_{k-1} - \varpi_k) : -k \leq t \leq 2n-2k+2 \}.
\] 
\end{proposition}

\begin{proof}
Let $L_1, L_2, \cdots, L_n$ be an orthonormal basis of $\mathfrak X(T) \otimes \mathbb R$. 
Then the set of all positive roots of $C_n$ consists of the following: for $1 \leq  i<j \leq n$,
$$ L_i-L_j, \quad L_i+L_j, \quad 2L_i.$$
Hence, $2\rho_G=\sum_{i<j}(L_i-L_j)+\sum_{i<j}(L_i+L_j)+ 2\sum_{i=1}^{n}L_i = \sum_{i=1}^{n}(2n-2i+2)L_i$.
On the other hand, since the reductive part $L$ of $P=P^{\alpha_{k-1}} \cap P^{\alpha_k}$ is of type $A_{k-2} \times C_{n-k}$, the positive roots of $L$ are $L_i-L_j$ for $1 \leq i<j \leq k-1$, $L_i \pm L_j$ for $k+1 \leq i<j \leq n$, and $2L_i$ for $i= k+1, \cdots , n$.
Therefore, we have $2\rho_L=\sum_{i=1}^{k-1}(k-2i)L_i+\sum_{i=k+1}^{n}(2n-2i+2)L_i$ and 
\begin{align*}
2\rho_G-2\rho_L & = \sum_{i=1}^{k}(2n-2i+2)L_i-\sum_{i=1}^{k-1}(k-2i)L_i \\
& = (2n-2k+2)\sum_{i=1}^{k}L_i+k\sum_{i=1}^{k-1}L_i=(2n-2k+2)\varpi_{k} + k\varpi_{k-1}.
\end{align*} 
The result follows from Lemma~\ref{moment polytope}.
\end{proof}

As $\alpha_n$ is a unique long simple root of the symplectic group $C_n=\Sp(2n, \mathbb C)$ and the other simple roots are short, 
it is necessary to separate the calculations in the two cases: 
(Case I) $k = n$, (Case II) $2 \leq k \leq n-1$. 

\begin{proposition} 
\label{GRLB_X3n} 
For $n \geq 2$, the greatest Ricci lower bound of an odd Lagrangian Grassmannian $X^3(n, n)$ is 
$R(X^3(n, n)) = \displaystyle \frac{2}{(n+2) \displaystyle \int_0^1 (1-t^2)^n \, dt} = \frac{2 \times (2n+1)!}{(n+2)(2^n \times n !)^2}.$ 
\end{proposition}

\begin{proof}
By Proposition~\ref{moment polytope_3}, 
the moment polytope $\Delta_3$ of $X^3(n, n)$ is parametrized as 
$$\gamma_3(t) = (n+t)\varpi_{n-1} + (2-t) \varpi_n$$ 
for $-n \leq t \leq 2$. 
Let's count the number of roots in $\Phi_{P^u}$ for each value of the coefficients of $\alpha_{n-1}, \alpha_n$:  
\begin{enumerate}
\item[(i)] the cardinality of the set $\{ \sum_{i=1}^n c_i \alpha_i \in \Phi^+ : c_{n-1} = 1, c_n = 0 \}$ is $n-1$, 
\item[(ii)] the cardinality of the set $\{ \sum_{i=1}^n c_i \alpha_i \in \Phi^+ : c_{n-1} = 0, c_n = 1 \}$ is $1$, 
\item[(iii)] the cardinality of the set $\{ \sum_{i=1}^n c_i \alpha_i \in \Phi^+ : c_{n-1} = 1, c_n = 1 \}$ is $n-1$, and 
\item[(iv)] the cardinality of the set $\{ \sum_{i=1}^n c_i \alpha_i \in \Phi^+ : c_{n-1} = 2, c_n = 1 \}$ is $\frac{n(n-1)}{2}$. 
\end{enumerate}
Then, from $\kappa (\alpha_{n-1}, \varpi_{n-1}) = 1$ and $\kappa (\alpha_n, \varpi_n) = 2$ we have 
\begin{align*}
P_{DH}(\gamma_3(t)) & = \prod_{\alpha \in \Phi_{P^u}} \kappa(\alpha, (n+t) \varpi_{n-1} + (2-t) \varpi_n) \\
& = (n+t)^{n-1} \cdot 2(2-t) \{ (n+t) + 2(2-t) \}^{n-1} \{ 2(n+t) + 2(2-t) \}^{\frac{n(n-1)}{2}} \\
& = 2(2n+4)^{\frac{n(n-1)}{2}} \cdot (n+t)^{n-1} (2-t) (n+4-t)^{n-1}.
\end{align*}
Thus, we get the volume of the moment polytope $\Delta_3$ with respect to the Duistermaat--Heckman measure 
\begin{align*}
\text{Vol}_{DH}(\Delta_3) & = \int_{-n}^{2} P_{DH}(\gamma_3(t)) \, dt = \int_{-n}^{2} 2(2n+4)^{\frac{n(n-1)}{2}} \cdot (n+t)^{n-1} (2-t) (n+4-t)^{n-1} \, dt \\
& = 2(2n+4)^{\frac{n(n-1)}{2}} \int_{-n}^{2} (2-t) \{ (n+t) (n+4-t)\}^{n-1} \, dt \\
& = 2(2n+4)^{\frac{n(n-1)}{2}} \int_{-n}^{2} (2-t) \{ -(t-2)^2 + (n+2)^2 \}^{n-1} \, dt \\
& = 2(2n+4)^{\frac{n(n-1)}{2}} \Big [ \frac{1}{2n}\{ -(t-2)^2 + (n+2)^2 \}^{n} \Big ]_{-n}^2 \\
& = 2(2n+4)^{\frac{n(n-1)}{2}} \cdot  \frac{(n+2)^{2n}}{2n} 
\end{align*}
As we know that 
$\displaystyle \int (2-t) \{ -(t-2)^2 + (n+2)^2 \}^{n-1} \, dt = \frac{1}{2n} \{ -(t-2)^2 + (n+2)^2 \}^{n}  + C$
from the previous calculation result, we can use integration by parts 
\begin{align*}
\int_{-n}^{2} (2-t) \cdot \frac{P_{DH}(\gamma_3(t))}{2(2n+4)^{\frac{n(n-1)}{2}}} \, dt 
& = \int_{-n}^{2} (2-t) \cdot (2-t) \{ -(t-2)^2 + (n+2)^2 \}^{n-1} \, dt \\ 
& = \Big [ (2-t) \cdot \frac{1}{2n}\{ -(t-2)^2 + (n+2)^2 \}^{n} \Big ]_{-n}^2 - \int_{-n}^{2} (-1) \cdot \frac{1}{2n} \{ -(t-2)^2 + (n+2)^2 \}^{n} \, dt \\
& = \frac{1}{2n} \int_{-n}^{2} \{ -(t-2)^2 + (n+2)^2 \}^{n} \, dt = \frac{(n+2)^{2n}}{2n} \int_{-n}^{2} \left \{ 1 - \left (\frac{2-t}{n+2} \right )^2 \right \}^{n} \, dt \\
& = \frac{(n+2)^{2n}}{2n} \int_{1}^{0}(1-s^2)^{n} \, (-(n+2) ds) = \frac{(n+2)^{2n}}{2n} \cdot (n+2) \int_{0}^{1}(1-s^2)^{n} \, ds
\end{align*}
Here, we use the substitution $s = \frac{2-t}{n+2}$ for the last line of equations. 
Thus, we have 
$$\displaystyle \frac{1}{\text{Vol}_{DH}(\Delta_3)} \int_{-n}^{2} (2-t) \cdot P_{DH}(\gamma_3(t)) \, dt = (n+2) \int_{0}^{1}(1-t^2)^{n} \, dt, $$ 
from which we get the barycenter 
$\displaystyle \textbf{bar}_{DH}(\Delta_3) = \gamma_3 \left( 2 - (n+2) \displaystyle \int_0^1 (1-t^2)^n \, dt \right)$ 
of the moment polytope $\Delta_3$ with respect to the Duistermaat--Heckman measure. 
Since $2 - (n+2) \displaystyle \int_0^1 (1-t^2)^n \, dt < 0$ by the first statement of Lemma~\ref{inequalty:X3n}, 
the greatest Ricci lower bound of $X^3(n)$ is 
$$R(X^3(n, n)) = \displaystyle \frac{2}{2 - \{2-(n+2) \int_0^1 (1-t^2)^n \, dt \} } = \frac{2}{(n+2) \int_0^1 (1-t^2)^n \, dt} = \frac{2 \times (2n+1)!}{(n+2)(2^n \times n !)^2}$$ 
by Corollary~\ref{formula for greatest Ricci lower bounds} and the second statement of Lemma~\ref{inequalty:X3n} .
\end{proof}

\begin{lemma} 
\label{inequalty:X3n} 
For any $n \geq 2$, the inequality $(n+2) \displaystyle \int_0^1 (1-t^2)^n \, dt > 2$ holds. 
Furthermore, we have $(n+2) \displaystyle \int_0^1 (1-t^2)^n \, dt = \frac{(n+2)(2^n \times n !)^2}{(2n+1)!}$ for any $n \geq 0$.  
\end{lemma}

\begin{proof}
Putting a sequence $a_n := (n+2) \displaystyle \int_0^1 (1-t^2)^n \, dt$ for $n \geq 0$, we will find a recurrence relation for $a_n$. 
Using integration by parts, we have
\begin{align*}
\int_0^1 (1-t^2)^{n+1} \, dt 
& = \Big [ t \cdot (1-t^2)^{n+1} \Big ]_{0}^1 - \int_0^1 t \cdot (n+1)(1-t^2)^{n}(-2t) \, dt  \\
& = 2(n+1) \int_0^1 t^2 (1-t^2)^{n} \, dt \\
& = \int_0^1 (1-t^2)^{n} \, dt - \int_0^1 t^2 (1-t^2)^{n} \, dt 
\end{align*}
from which we deduce that $\displaystyle \int_0^1 (1-t^2)^{n} \, dt = (2n+3) \int_0^1 t^2 (1-t^2)^{n} \, dt$. 
Then, 
\begin{align*}
a_{n+1} 
& = \{(n+1)+2\} \int_0^1 (1-t^2)^{n+1} \, dt = (n+3) \int_0^1 \{ (1-t^2)^{n} - t^2 (1-t^2)^{n} \} \, dt \\
& = (n+3) \left ( 1 - \frac{1}{2n+3} \right ) \int_0^1 (1-t^2)^{n} \, dt = \frac{n+3}{n+2} \cdot \frac{2n+2}{2n+3} \, a_n.
\end{align*}
From the recurrence relation, we obtain that 
\begin{align*}
a_{n+2} 
& = \frac{n+4}{n+3} \cdot \frac{2n+4}{2n+5} \, a_{n+1} = \frac{n+4}{n+2} \cdot \frac{(2n+2)(2n+4)}{(2n+3)(2n+5)} \, a_n, \\
a_{n+3} 
& = \frac{n+5}{n+4} \cdot \frac{2n+6}{2n+7} \, a_{n+2} = \frac{n+5}{n+2} \cdot \frac{(2n+2)(2n+4)(2n+6)}{(2n+3)(2n+5)(2n+7)} \, a_n.
\end{align*}
More generally, we see that $a_{n+p} = \displaystyle \frac{n+p+2}{n+2} \cdot \frac{(2n+2)(2n+4)(2n+6) \times \cdots \times (2n+2p)}{(2n+3)(2n+5)(2n+7) \times \cdots \times (2n+2p+1)} \, a_n $ for any integer $p \geq 0$.  
Substituting $n=0$, we get $a_{p} = \displaystyle \frac{p+2}{2} \cdot \frac{2 \times 4 \times 6 \times \cdots \times 2p}{3 \times 5 \times 7 \times \cdots \times (2p+1)} \, a_0 = (p+2) \cdot  \frac{2^p \times p !}{\displaystyle \frac{(2p+1)!}{2^p \times p !}}$. 
Thus, the general term of the sequence $a_n$ is obtained: $a_n = (n+2) \cdot \displaystyle \frac{(2^n \times n !)^2}{(2n+1)!}$. 

Since $a_{n+1} = \displaystyle \frac{n+3}{n+2} \cdot \frac{2n+2}{2n+3} \, a_n = \left ( 1 + \frac{n}{2n^2 + 7n + 6} \right ) \, a_n$ and 
$a_2 = 4 \displaystyle \int_0^1 (1-t^2)^2 \, dt = \frac{32}{15} > 2$, 
the inequalty $a_n >2$ holds for every natural number $n \geq 2$. 
\end{proof}

From the proof of Lemma~\ref{inequalty:X3n}, we see that the sequence $(n+2) \displaystyle \int_0^1 (1-t^2)^n \, dt$ is increasing as $n$ increases; 
hence the greatest Ricci lower bound $R(X^3(n, n))$ decreases. 
As specific examples, the approximate values of $R(X^3(n, n))$ for small $n$ are summarized as shown in the following table. 

\begin{table}[h]
\begin{center}
\begin{tabular}{c c c c c c c}
		\toprule
		$n$ & $2$   & $3$ & $4$ & $5$ & $6$ & $7$
\\
		\midrule 
	$R(X^3(n, n))$	&	$\displaystyle \frac{15}{16} = 0.9375$	&	$\displaystyle \frac{7}{8} = 0.875$  &	$\displaystyle \frac{105}{128} \approx 0.820$ 	&	$\displaystyle \frac{99}{128} \approx 0.773$ 	&	$\displaystyle \frac{3003}{4096} \approx 0.733$  &	$\displaystyle \frac{715}{1024} \approx 0.698$ 
\\
		\bottomrule
\end{tabular}
\caption{Approximate values of $R(X^3(n, n))$.}
\label{table3}
\end{center}
\end{table}

\vskip -1em 

\begin{corollary} 
\label{limit of R(X3n)}
The greatest Ricci lower bound $R(X^3(n, n))$ of $X^3(n, n)$ converges to $0$ as $n$ increases, that is, $\displaystyle \lim_{n \to \infty} R(X^3(n, n)) = 0$. 
In addition, Tian's alpha-invariant $\alpha(X^3(n, n))$ also converges to $0$ as $n$ increases.
\end{corollary}

\begin{proof}
Since the formula $R(X^3(n, n)) = \displaystyle \frac{2 \times (2n+1)!}{(n+2)(2^n \times n !)^2}$ in Proposition \ref{GRLB_X3n} is expressed as factorials, it is useful to use an approximation for them. 
Stirling's approximation gives bounds of $n!$ valid for all positive integers $n$:
$\sqrt{2 \pi n} \left ( \displaystyle \frac{n}{e} \right )^n \leq n ! \leq e \sqrt{n} \left ( \displaystyle \frac{n}{e} \right )^n$.
Thus, we get an inequality 
\begin{align*}
R(X^3(n, n)) = \frac{2 \times (2n+1)!}{(n+2)(2^n \times n !)^2} 
& \leq \frac{2 \times e \sqrt{2n+1} \left ( \displaystyle \frac{2n+1}{e} \right )^{2n+1}}{(n+2) \left \{ 2^n \times \sqrt{2 \pi n} \left ( \displaystyle \frac{n}{e} \right )^n \right \}^2} \\
& = \frac{2 \times \sqrt{2n+1} (2n+1)^{2n+1}}{(n+2) \times 2^{2n} \times 2 \pi n \times n^{2n}} 
= \frac{2 \sqrt{2n+1}}{\pi(n+2)} \cdot \left ( 1 + \frac{1}{2n} \right )^{2n+1}. 
\end{align*}
As $\displaystyle \lim_{n \to \infty} \left ( 1 + \frac{1}{2n} \right )^{2n+1} = e$ and $\displaystyle \lim_{n \to \infty} \frac{\sqrt{2n+1}}{n+2} = 0$, 
we conclude that $\displaystyle \lim_{n \to \infty} R(X^3(n, n)) = 0$. 

On the other hand, Tian \cite{tian87} showed that a lower bound of $R(X)$ in terms of the alpha-invariant $\alpha(X)$ for $X$ which does not admit a K\"{a}hler--{E}instein metric: $R(X) \geq \alpha(X) \cdot \frac{\dim X + 1}{\dim X}$. 
Hence, $\dim X^3(n, n) = \frac{n(n+3)}{2}$ implies that $\displaystyle \lim_{n \to \infty} \alpha(X^3(n, n)) = 0$. 
\end{proof}

Now, let's move on to the next case for $2 \leq k \leq n-1$. 

\begin{proposition} 
\label{GRLB_X3} 
For $n > k \geq 2$, the greatest Ricci lower bound of $X^3(n, k)$ is 
$$R(X^3(n, k)) = \displaystyle \frac{(2n-2k+2) \displaystyle \int_{-k}^{2n-2k+2} (k+t)^{k-1} (2n-2k+2-t)^{2n-2k+1} (4n-3k+4-t)^{k-1} \, dt}{\displaystyle\int_{-k}^{2n-2k+2} (k+t)^{k-1} (2n-2k+2-t)^{2n-2k+2} (4n-3k+4-t)^{k-1} \, dt}.$$ 
\end{proposition}

\begin{proof}
By Proposition \ref{moment polytope_3}, 
the moment polytope $\Delta_3$ of $X^3(n, k)$ is parametrized as 
$$\gamma_3(t) = (k+t)\varpi_{k-1} + (2n-2k+2-t) \varpi_k$$ 
for $-k \leq t \leq 2n-2k+2$. 
Let's count the number of roots in $\Phi_{P^u}$ for each value of the coefficients of $\alpha_{k-1}, \alpha_k$:  
\begin{enumerate}
\item[(i)] the cardinality of the set $\{ \sum_{i=1}^n c_i \alpha_i \in \Phi^+ : c_{k-1} = 1, c_k = 0 \}$ is $k-1$, 
\item[(ii)] the cardinality of the set $\{ \sum_{i=1}^n c_i \alpha_i \in \Phi^+ : c_{k-1} = 0, c_k = 1 \}$ is $2(n-k)$, 
\item[(iii)] the cardinality of the set $\{ \sum_{i=1}^n c_i \alpha_i \in \Phi^+ : c_{k-1} = 1, c_k = 1 \}$ is $2(k-1)(n-k)$, 
\item[(iv)] the cardinality of the set $\{ \sum_{i=1}^n c_i \alpha_i \in \Phi^+ : c_{k-1} = 0, c_k = 2 \}$ is $1$,  
\item[(v)] the cardinality of the set $\{ \sum_{i=1}^n c_i \alpha_i \in \Phi^+ : c_{k-1} = 1, c_k = 2 \}$ is $k-1$, and 
\item[(vi)] the cardinality of the set $\{ \sum_{i=1}^n c_i \alpha_i \in \Phi^+ : c_{k-1} = 2, c_k = 2 \}$ is $\frac{k(k-1)}{2}$. 
\end{enumerate}
Then, from $\kappa (\alpha_{k-1}, \varpi_{k-1}) = 1$ and $\kappa (\alpha_k, \varpi_k) = 1$ we have 
\begin{align*}
P_{DH}(\gamma_3(t)) & = \prod_{\alpha \in \Phi_{P^u}} \kappa(\alpha, (k+t) \varpi_{k-1} + (2n-2k+2-t) \varpi_k) \\
& = (k+t)^{k-1} (2n-2k+2-t)^{2(n-k)} \{(k+t) + (2n-2k+2-t) \}^{2(k-1)(n-k)}  \\
& \times \{ 2 (2n-2k+2-t) \} \{ (k+t) + 2 (2n-2k+2-t) \}^{k-1} \{ 2 (k+t) + 2 (2n-2k+2-t) \}^{\frac{k(k-1)}{2}}\\
& = 2(2n-k+2)^{2(k-1)(n-k)} \{ 2 (2n-k+2) \}^{\frac{k(k-1)}{2}}
(k+t)^{k-1} (2n-2k+2-t)^{2n-2k+1} (4n-3k+4-t)^{k-1}.
\end{align*}
As $\displaystyle \frac{1}{\text{Vol}_{DH}(\Delta_3)} \int_{-k}^{2n-2k+2} (k+t) \cdot P_{DH}(\gamma_3(t)) \, dt 
= \frac{ \int_{-k}^{2n-2k+2} (k+t)^{k} (2n-2k+2-t)^{2n-2k+1} (4n-3k+4-t)^{k-1} \, dt}{\int_{-k}^{2n-2k+2} (k+t)^{k-1} (2n-2k+2-t)^{2n-2k+1} (4n-3k+4-t)^{k-1} \, dt}$, 
we get the barycenter 
$\displaystyle \textbf{bar}_{DH}(\Delta_3) 
= \gamma_3 \left( \frac{ \int_{-k}^{2n-2k+2} (k+t)^{k} (2n-2k+2-t)^{2n-2k+1} (4n-3k+4-t)^{k-1} \, dt}{\int_{-k}^{2n-2k+2} (k+t)^{k-1} (2n-2k+2-t)^{2n-2k+1} (4n-3k+4-t)^{k-1} \, dt} - k \right)$
of the moment polytope $\Delta_3$ with respect to the Duistermaat--Heckman measure.  

Since $\displaystyle \frac{ \int_{-k}^{2n-2k+2} (k+t)^{k} (2n-2k+2-t)^{2n-2k+1} (4n-3k+4-t)^{k-1} \, dt}{\int_{-k}^{2n-2k+2} (k+t)^{k-1} (2n-2k+2-t)^{2n-2k+1} (4n-3k+4-t)^{k-1} \, dt} - k < 0$
by Lemma \ref{inequalty:X3nk}, 
the greatest Ricci lower bound of $X^3(n, k)$ is 
\begin{align*}
R(X^3(n, k)) & = \displaystyle \frac{2n-2k+2}{(2n-2k+2) - \left (\displaystyle \frac{ \int_{-k}^{2n-2k+2} (k+t)^{k} (2n-2k+2-t)^{2n-2k+1} (4n-3k+4-t)^{k-1} \, dt}{\int_{-k}^{2n-2k+2} (k+t)^{k-1} (2n-2k+2-t)^{2n-2k+1} (4n-3k+4-t)^{k-1} \, dt} - k \right ) } \\
& = \displaystyle \frac{2n-2k+2}{(2n-k+2) - \displaystyle \frac{ \int_{-k}^{2n-2k+2} (k+t)^{k} (2n-2k+2-t)^{2n-2k+1} (4n-3k+4-t)^{k-1} \, dt}{\int_{-k}^{2n-2k+2} (k+t)^{k-1} (2n-2k+2-t)^{2n-2k+1} (4n-3k+4-t)^{k-1} \, dt} } \\
& = \displaystyle \frac{(2n-2k+2)\int_{-k}^{2n-2k+2} (k+t)^{k-1} (2n-2k+2-t)^{2n-2k+1} (4n-3k+4-t)^{k-1} \, dt}{\int_{-k}^{2n-2k+2} (k+t)^{k-1} (2n-2k+2-t)^{2n-2k+2} (4n-3k+4-t)^{k-1} \, dt} 
\end{align*}
by Corollary~\ref{formula for greatest Ricci lower bounds}. 
\end{proof}

\begin{lemma} 
\label{inequalty:X3nk} 
For any $n > k \geq 2$, the inequality 
$$\displaystyle \frac{ \displaystyle \int_{-k}^{2n-2k+2} (k+t)^{k} (2n-2k+2-t)^{2n-2k+1} (4n-3k+4-t)^{k-1} \, dt}{\displaystyle \int_{-k}^{2n-2k+2} (k+t)^{k-1} (2n-2k+2-t)^{2n-2k+1} (4n-3k+4-t)^{k-1} \, dt} < k $$ holds. 
\end{lemma}

\begin{proof}
It suffices to show that 
$\displaystyle \int_{-k}^{2n-2k+2} t (k+t)^{k-1} (2n-2k+2-t)^{2n-2k+1} (4n-3k+4-t)^{k-1} \, dt < 0$ for $n > k \geq 2$. 
Putting the left side of this inequality as $I(n,k)$, we have 
\begin{align*}
I(n,k) 
& = \int_{-k}^{2n-2k+2} t (2n-2k+2-t)^{2n-2k+1} \{ (k+t)(4n-3k+4-t) \}^{k-1} \, dt  \\
& = \int_{-k}^{2n-2k+2} t (2n-2k+2-t)^{2n-2k} \cdot (2n-2k+2-t) \{ -(2n-2k+2-t)^2 + (2n-k+2)^2 \}^{k-1} \, dt.
\end{align*}
Since $\displaystyle \int (2n-2k+2-t) \{ -(2n-2k+2-t)^2 + (2n-k+2)^2 \}^{k-1} \, dt = \frac{1}{2k} \{ -(2n-2k+2-t)^2 + (2n-k+2)^2 \}^{k} + C$, 
we can use integration by parts repetitively: 
\begin{align*}
I(n,k) 
& = \Big [ t (2n-2k+2-t)^{2n-2k} \cdot \frac{1}{2k} \{ -(2n-2k+2-t)^2 + (2n-k+2)^2 \}^{k} \Big ]_{-k}^{2n-2k+2}\\
& \quad - \int_{-k}^{2n-2k+2} (2n-2k+2-t)^{2n-2k-1} \{(2n-2k+2) - (2n-2k+1) t \} \\
& \quad \times \frac{1}{2k} \{ -(2n-2k+2-t)^2 + (2n-k+2)^2 \}^{k} \, dt \\
& = \frac{1}{2k} \int_{-k}^{2n-2k+2} (2n-2k+2-t)^{2n-2k-2} \{ (2n-2k+1) t - (2n-2k+2) \} \\
& \quad \times (2n-2k+2-t) \{ -(2n-2k+2-t)^2 + (2n-k+2)^2 \}^{k} \, dt \\
& = \frac{1}{2k} \Big [\{ (2n-2k+1) t - (2n-2k+2) \}(2n-2k+2-t)^{2n-2k-2} \\
& \quad \times \frac{1}{2(k+1)} \{ -(2n-2k+2-t)^2 + (2n-k+2)^2 \}^{k+1} \Big ]_{-k}^{2n-2k+2} - \frac{1}{2k} \int_{-k}^{2n-2k+2} (2n-2k+2-t)^{2n-2k-3} \\
& \quad \times [ -(2n-2k+1)(2n-2k-1) t + (2n-2k+2)\{ 1+ (2n-2k)+(2n-2k-2) \} ] \\
& \quad \times \frac{1}{2(k+1)} \{ -(2n-2k+2-t)^2 + (2n-k+2)^2 \}^{k+1} \, dt.
\end{align*}
If $n-k = 2$, then we have 
\begin{align*}
I(n,k) 
& = \frac{1}{2k} \cdot \frac{1}{2(k+1)} 
\int_{-k}^{2n-2k+2} [ (2n-2k+1)(2n-2k-1) t - (2n-2k+2)\{ 1+ (2n-2k)+(2n-2k-2) \} ] \\
& \quad \times (2n-2k+2-t) \{ -(2n-2k+2-t)^2 + (2n-k+2)^2 \}^{k+1} \, dt  \\ 
& = \frac{1}{2k} \cdot \frac{1}{2(k+1)} \cdot \frac{1}{2(k+2)} \Big [ (2n-2k+2)(2n-2k)(2n-2k-2)(2n-k+2)^{2(k+2)} \\
& \quad -(2n-2k+1)(2n-2k-1) \int_{-k}^{2n-2k+2} \{ -(2n-2k+2-t)^2 + (2n-k+2)^2 \}^{k+2} \, dt \Big ].
\end{align*}
As a result, we obtain that 
\begin{align*}
I(n,k) 
& = \frac{1}{2k} \cdot \frac{1}{2(k+1)} \times \cdots \times \frac{1}{2n} 
\Big [ (2n-2k+2)(2n-2k)(2n-2k-2) \times \cdots \times 4 \times 2 \times (2n-k+2)^{2n} \\
& - (2n-2k+1)(2n-2k-1)(2n-2k-3) \times \cdots \times 3 \times \int_{-k}^{2n-2k+2} \{ -(2n-2k+2-t)^2 + (2n-k+2)^2 \}^{n} \, dt \Big ].
\end{align*}
For a fixed $\ell := n-k$, if we show the following inequality, then we get the conclusion we want:  
$$ (n+\ell+2) \int_{0}^{1} (1-s^2)^n \, ds > \frac{(2\ell+2)(2\ell)(2\ell-2) \times \cdots \times 4 \times 2}{(2\ell+1)(2\ell-1)(2\ell-3) \times \cdots \times 3} \qquad \text{for $n \geq \ell +2$.}$$

We give a proof by induction on $n$. \\ 
(i) First, the statement holds for $n = \ell +2$ because $(2\ell+4) \displaystyle \int_{0}^{1} (1-s^2)^{\ell + 2} \, ds > \frac{(2\ell+2)(2\ell)(2\ell-2) \times \cdots \times 4 \times 2}{(2\ell+1)(2\ell-1)(2\ell-3) \times \cdots \times 3}$ for any $\ell \geq 1$. 
Indeed, we can show that the ratio 
$f(\ell):= \displaystyle \frac{(2\ell+4) \displaystyle \int_{0}^{1} (1-s^2)^{\ell + 2} \, ds}{\displaystyle \frac{(2\ell+2)(2\ell)(2\ell-2) \times \cdots \times 4 \times 2}{(2\ell+1)(2\ell-1)(2\ell-3) \times \cdots \times 3}}$ is greater than one. 
Since $\displaystyle \int_0^1 (1-s^2)^{\ell+3} \, ds = \frac{2\ell+6}{2\ell+7} \int_0^1 (1-s^2)^{\ell+2} \, ds$ from the proof of Lemma \ref{inequalty:X3n}, 
we can express each element of the sequence $f(\ell)$ as a function of the preceding ones:
\begin{align*}
f(\ell+1) 
& = \displaystyle \frac{(2\ell+6) \displaystyle \int_{0}^{1} (1-s^2)^{\ell + 3} \, ds}{\displaystyle \frac{(2\ell+4)(2\ell+2)(2\ell) \times \cdots \times 4 \times 2}{(2\ell+3)(2\ell+1)(2\ell-1) \times \cdots \times 3}} 
= \displaystyle \frac{(2\ell+6) \cdot \displaystyle \frac{2\ell+6}{2\ell+7} \int_0^1 (1-s^2)^{\ell+2} \, ds}{\displaystyle \frac{2\ell+4}{2\ell+3} \cdot \displaystyle \frac{(2\ell+2)(2\ell) \times \cdots \times 4 \times 2}{(2\ell+1)(2\ell-1) \times \cdots \times 3}} \\
& = \displaystyle \frac{(\ell+3)^2 (2\ell+3)}{(\ell + 2)^2(2\ell+7)} f(\ell). 
\end{align*}
From the recurrence relation, we obtain that 
\begin{align*}
f(\ell+2) 
& = \displaystyle \frac{(\ell+4)^2 (2\ell+5)}{(\ell + 3)^2(2\ell+9)} f(\ell+1) 
= \frac{(\ell+4)^2 (2\ell+5)}{(\ell + 3)^2(2\ell+9)} \cdot \frac{(\ell+3)^2 (2\ell+3)}{(\ell + 2)^2(2\ell+7)} f(\ell) \\
& = \frac{(\ell+4)^2 (2\ell+3) (2\ell+5)}{(\ell + 2)^2 (2\ell+7)(2\ell+9)} f(\ell), \\
f(\ell+3) 
& = \displaystyle \frac{(\ell+5)^2 (2\ell+7)}{(\ell + 4)^2(2\ell+11)} f(\ell+2) 
= \frac{(\ell+5)^2 (2\ell+7)}{(\ell + 4)^2(2\ell+11)} \cdot \frac{(\ell+4)^2 (2\ell+3) (2\ell+5)}{(\ell + 2)^2 (2\ell+7)(2\ell+9)} f(\ell)  \\
& = \frac{(\ell+5)^2 (2\ell+3) (2\ell+5)}{(\ell + 2)^2 (2\ell+9)(2\ell+11)} f(\ell). 
\end{align*}
More generally, we see that $f(\ell+p) = \displaystyle \frac{(\ell+p+2)^2 (2\ell+3) (2\ell+5)}{(\ell + 2)^2 (2\ell+2p+3)(2\ell+2p+5)} f(\ell)$ for any natural number $p$.  
Thus, 
$f(p+1) = \displaystyle \frac{(p+3)^2 \times 5 \times 7}{3^2 (2p+5)(2p+7)} f(1) = \frac{35 (p+3)^2}{9 (2p+5)(2p+7)} \times \frac{36}{35} = \frac{4(p+3)^2}{(2p+5)(2p+7)}= 1+ \frac{1}{4 p^2 + 24p +35} >1$. 
(ii) Assume the induction hypothesis that 
$(m+\ell+2) \displaystyle \int_{0}^{1} (1-s^2)^m \, ds > \displaystyle \frac{(2\ell+2)(2\ell)(2\ell-2) \times \cdots \times 4 \times 2 }{(2\ell+1)(2\ell-1)(2\ell-3) \times \cdots \times 3}$ for a particular $m \geq \ell +2$. 
Since we know that $\displaystyle \int_0^1 (1-s^2)^{m} \, ds = (2m+3) \int_0^1 s^2 (1-s^2)^{m} \, ds$ from the proof of Lemma \ref{inequalty:X3n}, 
\begin{align*}
((m+1)+\ell+2) \displaystyle \int_{0}^{1} (1-s^2)^{m+1} \, ds
& = (m+\ell+3) \cdot \frac{2m+2}{2m+3} \displaystyle \int_{0}^{1} (1-s^2)^{m} \, ds\\
& = \frac{m+\ell+3}{m+\ell+2} \cdot \frac{2m+2}{2m+3} \cdot (m+\ell+2) \displaystyle \int_{0}^{1} (1-s^2)^{m} \, ds\\
& = \Bigg \{ 1+ \frac{m-\ell}{2m^2 + (2 \ell+7) m + 3(\ell+2)} \Bigg \} \cdot (m+\ell+2) \displaystyle \int_{0}^{1} (1-s^2)^{m} \, ds.
\end{align*}
Therefore, if $n-\ell \geq 2$ then we have  
\begin{align*} 
((m+1)+\ell+2) \displaystyle \int_{0}^{1} (1-s^2)^{m+1} \, ds 
& >(m+\ell+2) \displaystyle \int_{0}^{1} (1-s^2)^{m} \, ds \\
& > \displaystyle \frac{(2\ell+2)(2\ell)(2\ell-2) \times \cdots \times 4 \times 2 }{(2\ell+1)(2\ell-1)(2\ell-3) \times \cdots \times 3}. 
\end{align*} 

Since both the base case and the inductive step have been proved as true, by mathematical induction the inequalty holds for every natural number $n \geq \ell +2$. 
\end{proof}

\begin{remark}
The formula for the greatest Ricci lower bound of $X^3(n, n)$ in Proposition \ref{GRLB_X3n} is consistent with the formula in Proposition \ref{GRLB_X3} 
when $k = n$ is substituted in it. 
\end{remark}

Interestingly, we observe that for a fixed integer $k$ the value $R(X^3(n, k))$ always approaches to $1$ asymptotically as $n$ increases. 

\begin{corollary} 
\label{limit of R(X3nk)}
For a fixed integer $k \geq 2$, the greatest Ricci lower bound $R(X^3(n, k))$ of $X^3(n, k)$ converges to $1$ as $n$ increases, that is, $\displaystyle \lim_{n \to \infty} R(X^3(n, k)) = 1$. 
\end{corollary}

\begin{proof}
Recall that $R(X^3(n, k)) = \displaystyle \frac{2n-2k+2}{2n-k+2 - \displaystyle \frac{ \int_{-k}^{2n-2k+2} (k+t)^{k} (2n-2k+2-t)^{2n-2k+1} (4n-3k+4-t)^{k-1} \, dt}{\int_{-k}^{2n-2k+2} (k+t)^{k-1} (2n-2k+2-t)^{2n-2k+1} (4n-3k+4-t)^{k-1} \, dt} }$  
from Proposition \ref{GRLB_X3}.
Since all functions $k+t$, $2n-2k+2-t$ and $4n-3k+4-t$ in the integrands 
are positive for $-k < t < 2n-2k+2$, 
we get an inequality $R(X^3(n, k)) > \displaystyle \frac{2n-2k+2}{2n-k+2}$; 
hence for a fixed $k$ we have $\displaystyle \lim_{n \to \infty} R(X^3(n, k)) \geq \lim_{n \to \infty} \frac{2n-2k+2}{2n-k+2} = 1$. 
Therefore, we conclude that $\displaystyle \lim_{n \to \infty} R(X^3(n, k)) = 1$. 
\end{proof}

We conclude this paper with a remark on Question~\ref{realization}. 

\begin{corollary}\label{lowerbound}
For a given real number $t$ with $0 < t < 1$, we can find a sequence $k_n$ such that 
$$ \displaystyle \lim_{n \to \infty} R(X^3(n, k_n)) \geq t.$$
\end{corollary}

\begin{proof}
Let $r = \frac{2(t-1)}{t-2}$ and $k_n = [nr]$. 
Note that it suffices to consider only large enough natural numbers $n$.  
For $n \geq \frac{2}{r}$, we have $k_n \geq 2$, so $X^3(n,k_n)$ is well defined. 
Since $k_n \leq nr < k_n+1$, we have $\displaystyle \lim_{n \to \infty} \frac{k_n}{n} = r.$ 
By the proof of Corollary~\ref{limit of R(X3nk)}, we have 
$$\displaystyle \lim_{n \to \infty} R(X^3(n, k_n)) \geq \lim_{n \to \infty} \frac{2n-2k_n+2}{2n-k_n+2} = \frac{2-2r}{2-r} = t.$$
\end{proof}

\vskip 1em

\noindent
\textbf{Acknowledgements}. 
K.-D. Park would like to express his gratitude to Jihun Park for helpful discussions on K-stability of Fano manifolds. The authors would like to thank Yanir Rubinstein for valuable comments, and anonymous referee for useful comments and suggestions which lead to Question~\ref{realization} and Corollary~\ref{lowerbound}. 

\vskip 1em

\noindent
\textbf{Funding}. 
D. Hwang was supported by the Samsung Science and Technology Foundation under Project SSTF-BA1602-03, the National Research Foundation of Korea(NRF) grant funded by the Korea government(MSIT) (2021R1A2C1093787), and the Institute for Basic Science (IBS-R032-D1).  S.-Y. Kim was supported by the Institute for Basic Science (IBS-R003-D1).  
K.-D. Park was supported by the National Research Foundation of Korea(NRF) grant funded by the Korea government(MSIT) (No. 2019R1A2C3010487, 2021R1C1C2092610). 

\vskip 1em

\noindent
\textbf{Data availability}. 
Data sharing not applicable to this article as no datasets were generated or analysed during the current study.

\vskip 2em

\noindent
\textbf{\Large Declarations}

\vskip 1em

\noindent
\textbf{Competing interests}. 
The authors have no competing interests to declare that are relevant to the content of this article.

\vskip 3em


\providecommand{\bysame}{\leavevmode\hbox to3em{\hrulefill}\thinspace}
\providecommand{\MR}{\relax\ifhmode\unskip\space\fi MR }
\providecommand{\MRhref}[2]{%
  \href{http://www.ams.org/mathscinet-getitem?mr=#1}{#2}
}
\providecommand{\href}[2]{#2}

\end{document}